%% file: negsteps.tex
\documentclass[11pt,letterpaper]{article}
\input{sections/preamble}

\usepackage[margin=1in]{geometry}

\makeatletter
\def\blfootnote{\gdef\@thefnmark{}\@footnotetext}
\makeatother

\begin{document}

\title{Negative Stepsizes Make Gradient-Descent-Ascent Converge}

\author{
 	Henry Shugart
 	\\	UPenn \\	\texttt{hshugart@upenn.edu}
 	\and
 	Jason M. Altschuler
 	\\	UPenn \\	\texttt{alts@upenn.edu}
}
\date{} 
\maketitle

\input{sections/abstract}

\newpage
\setcounter{tocdepth}{2}
\tableofcontents
 \newpage

\input{sections/intro}

\input{sections/prelim}

\input{sections/linear}

\input{sections/nonlinear}

\input{sections/connections}

\input{sections/discussion}

\section*{Acknowledgements} We are grateful to Surbhi Goel, Aryan Mokhtari, Pablo Parrilo, Molei Tao, Yuanhao Wang, and David Zhang for insightful conversations about the related literature, and to the anonymous reviewers for helpful comments which have improved the exposition, in particular for suggesting to present our analysis for the nonlinear setting (in~\cref{app:2-step}) in the more accessible way that is now in~\cref{ssec:2step}. JMA acknowledges funding from a Sloan Research Fellowship and a Seed Grant Award from Apple.

\newpage
\appendix

\input{sections/app_alternative}

\input{sections/app_deferred}

\footnotesize
\addcontentsline{toc}{section}{References}
\bibliographystyle{plainnat}
\bibliography{negsteps}

\end{document}

%% file: sections/preamble.tex
\usepackage{amsmath,amssymb,amsthm}
\usepackage{MnSymbol} 
\usepackage{bm, bbm} %
\usepackage{graphicx,color}
\usepackage[hyphens]{url}
\usepackage{dsfont}
\usepackage{booktabs}
\usepackage[normalem]{ulem}
\usepackage{mathtools}
\usepackage[hidelinks,colorlinks=true,linkcolor=blue,citecolor=blue]{hyperref}
\usepackage[nameinlink,capitalize]{cleveref}
\usepackage{multirow}
\usepackage{algorithm}
\usepackage[noend]{algpseudocode}
\usepackage{tikz} %
\usepackage{pifont}
\usepackage{hhline}
\usepackage{makecell}
\usepackage{enumitem}

\crefname{section}{\mbox{\S\!\!}}{\mbox{\S\!\!}}

\DeclareRobustCommand{\circled}[1]{\relax%
  \ifmmode
    \mathord{\text{\ding{\numexpr171+#1\relax}}}%
  \else
    \ding{\numexpr171+#1\relax}%
  \fi
}

\crefname{figure}{Figure}{Figures}

\usepackage{soul} %
\usepackage[square,sort,numbers]{natbib}
\usepackage[sort,nocompress,noadjust]{cite}

\usepackage{subcaption}

\usepackage{tabularx}
\newcolumntype{L}{>{\raggedright\arraybackslash}X}

\numberwithin{equation}{section}

\usepackage{aliascnt} 

\numberwithin{equation}{section}

\newtheorem{theorem}{Theorem}[section]
\crefname{theorem}{Theorem}{Theorems}

\newcommand{\newaliasedtheorem}[3]{%
  \newaliascnt{#1}{theorem}%
  \newtheorem{#1}[#1]{#2}%
  \aliascntresetthe{#1}%
  \crefname{#1}{#2}{#3}%
}

\newaliasedtheorem{cor}{Corollary}{Corollaries}
\newaliasedtheorem{lemma}{Lemma}{Lemmas}
\newaliasedtheorem{remark}{Remark}{Remarks}
\newaliasedtheorem{prop}{Proposition}{Propositions}
\newaliasedtheorem{obs}{Observation}{Observations}
\newaliasedtheorem{problem}{Problem}{Problems}
\newaliasedtheorem{question}{Question}{Questions}
\newaliasedtheorem{example}{Example}{Examples}
\newaliasedtheorem{conj}{Conjecture}{Conjectures}
\newaliasedtheorem{defin}{Definition}{Definitions}

\newcommand{\cL}{\mathcal{L}}

\newcommand{\cO}{\mathcal{O}}
\newcommand{\cP}{\mathcal{P}}

\newcommand{\cT}{\mathcal{T}}

\newcommand{\R}{\mathbb{R}}

\newcommand{\N}{\mathbb{N}}

\newcommand*{\E}{\mathbb{E}}

\providecommand{\spann}{\operatorname{Span}}

\providecommand{\diag}{\operatorname{\mathbb{D}}}

\newcommand*{\eps}{\varepsilon}
\renewcommand{\epsilon}{\varepsilon}

\newcommand*{\sig}{\sigma}

\newcommand{\plusminus}{\raisebox{.2ex}{$\scriptstyle\pm$}}
\DeclareMathOperator*{\argmin}{argmin}

\providecommand{\abs}[1]{\left\lvert#1\right\rvert}

\providecommand{\norm}[1]{\lVert{#1}\rVert}

\newcommand{\T}{^\top}
\newcommand{\Ker}{\mathrm{Ker}}
\newcommand{\GDA}{\textsc{GDA}}

%% file: sections/abstract.tex
\begin{abstract}
  Efficient computation of min-max problems is a central question in optimization, learning, games, and control. Arguably the most natural algorithm is gradient-descent-ascent (GDA). However, since the 1970s, conventional wisdom has argued that GDA fails to converge even on simple problems. This failure spurred an extensive literature on modifying GDA with additional building blocks such as extragradients, optimism, momentum, anchoring, etc.\ In contrast, we show that GDA converges in its original form by simply using a judicious choice of stepsizes.

  The key innovation is the proposal of unconventional stepsize schedules (dubbed \emph{slingshot stepsize schedules}) that are time-varying, asymmetric, and periodically negative. We show that all three properties are necessary for convergence, and that altogether this enables GDA to converge on the classical counterexamples (e.g., unconstrained convex-concave problems). The core algorithmic intuition is that although negative stepsizes make backward progress, they de-synchronize the min and max variables (overcoming the cycling issue of GDA), and lead to a \emph{slingshot phenomenon} in which the forward progress in the other iterations is overwhelmingly larger. This results in fast overall convergence. Crucially, for this de-synchronization we break symmetry by alternating negative steps for the min and max variables---this goes beyond what is possible with classical reductions from min-max problems to monotone operator theory, highlighting the importance of the intrinsic asymmetry in min-max problems. 

  Geometrically, the slingshot dynamics leverage the non-reversibility of gradient flow: positive/negative steps cancel to first order, yielding a second-order net movement in a new direction that leads to convergence and is otherwise impossible for GDA to move in. We interpret this as a second-order finite-differencing algorithm and show that, intriguingly, it approximately implements consensus optimization, an empirically popular algorithm for min-max problems involving deep neural networks (e.g., training GANs).

\end{abstract}

%% file: sections/intro.tex
\section{Introduction}\label{sec:intro}
This paper revisits the numerical computation of min-max problems, a.k.a.\ saddle-point problems: 
\begin{equation}\label{eq:minmax}
\min_{x \in \mathcal{X}}\ \max_{y \in \mathcal{Y}}\; f(x,y).
\end{equation}
Efficient computation of such problems is a central question in many areas, such as applied mathematics~\citep{benzi2005numerical}, game theory~\citep{von1947theory}, constrained optimization \citep{Bertsekas_2014}, distributed optimization \citep{boyd2011distributed},  robust optimization~\citep{ben2009robust}, robust control~\citep{hast2013pid}, adversarial machine learning~\citep{Madry_Makelov_Schmidt_Tsipras_Vladu_2019}, and generative adversarial networks (GANs)~\citep{Goodfellow_Pouget-Abadie_Mirza_Xu_Warde-Farley_Ozair_Courville_Bengio_2014}, among many others.

\begin{figure}
	\centering
	\includegraphics[width=0.4\linewidth]{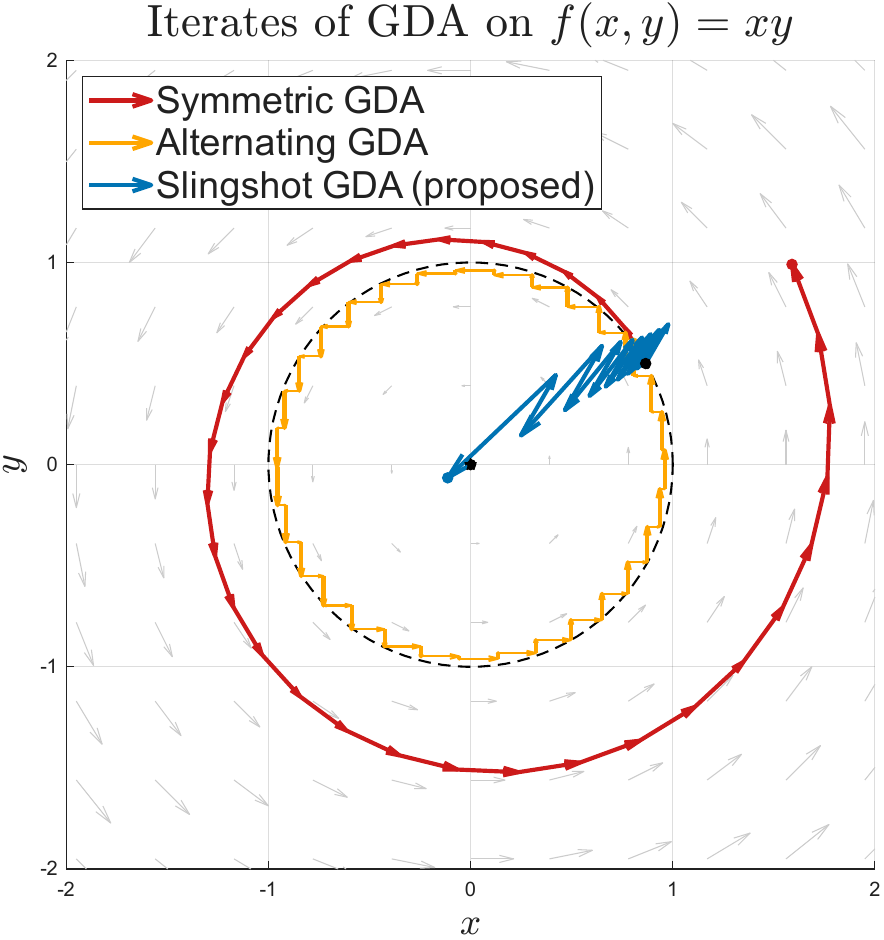}
	\caption{\footnotesize GDA trajectories for a classical counterexample: the unconstrained bilinear problem $\min_{x \in \R} \max_{y \in \R} xy$. The unique solution is the origin. GDA fails to converge with standard stepsize schedules, e.g., it diverges if $\alpha_t=\beta_t > 0$ (red), or limit-cycles if $\alpha_t,\beta_t$ are alternately positive and zero (yellow). Our proposed stepsize schedule (blue) enables GDA to converge for the first time. Standard GDA attempts to follow the grey vector field $(-\nabla_x f, \nabla_y f) = (-y,x)$; in contrast, our proposed GDA alternately moves in the directions $(\nabla_x f, \nabla_y f) = (y,x)$ and $(-\nabla_x f, -\nabla_y f) = (-y,-x)$. This trajectory ``slingshots'' to make progress every two iterations. 
    }
	\label{fig:GDAxy}
\end{figure}

First-order algorithms are the predominant choice for solving~\eqref{eq:minmax} due to their scalability, simplicity, and flexibility. Arguably the most natural such algorithm is gradient-descent-ascent (GDA), which is the extension of gradient descent from standard optimization to min-max optimization:
\begin{equation}\label{eq:GDAupdate}
	\begin{aligned}
		x_{t+1} &= x_t - \alpha_t \nabla_x f(x_t,y_t),\\
		y_{t+1} &= y_t + \beta_t \nabla_y f(x_t,y_t).
	\end{aligned}
\end{equation}
However, conventional wisdom states that GDA fails to converge even on simple problems. A folklore counterexample~\citep{samuelson1949market, korpelevich1976extragradient} is that GDA either diverges or gets stuck in limit cycles even on the toy $1$-dimensional unconstrained bilinear problem
\begin{equation}\label{eq:xy}
	\min_{x\in \R} \max_{y\in \R} \; xy\,.
\end{equation}
Indeed, for this problem, the GDA update~\eqref{eq:GDAupdate} simplifies to a linear dynamical system 
\begin{equation}\label{eq:xy-gda-intro}
    \begin{bmatrix}
        x_{t+1} \\ y_{t+1}
    \end{bmatrix}
    =
    \begin{bmatrix}
        1 & -\alpha_t \\ 
        \beta_t & 1
    \end{bmatrix}
    \begin{bmatrix}
        x_t \\ y_t
    \end{bmatrix}
\end{equation} 
that is expansive since the determinant of the update matrix is $1 + \alpha_t \beta_t$, which is at least $1$ for any standard choice of GDA stepsizes $\alpha_t,\beta_t \geq 0$. See~\cref{fig:GDAxy} for an illustration, and see \cref{ssec:warmup:diverge} for full details on this counterexample. 

\par In fact, this failure is generic: the GDA trajectory fails to converge on \emph{every} bilinear problem, not just this toy problem~\eqref{eq:xy}. Moreover, this failure provably applies to arbitrary stepsize schedules $\alpha_t,\beta_t \in \R$ which fall under any of the standard categories: non-negative stepsizes ($\alpha_t,\beta_t \geq 0$),  symmetric stepsizes ($\alpha_t = \beta_t$), and/or constant stepsizes ($\alpha_t \equiv \alpha, \beta_t \equiv \beta)$.

\par There are only a few restricted settings where GDA converges: convex-concave problems with uniformly bounded subgradients \citep{Nedić_Ozdaglar_2009,zinkevich2003online}, and under strong growth conditions such as strong convexity and strong concavity (albeit at a provably slow convergence rate~\citep{Liang_Stokes_2019,Lee_Cho_Yun_2024}). These assumptions prohibit natural settings like unconstrained bilinear, quadratic, or convex-concave problems.
Moreover, the former result only applies to the \emph{average} iterate---since even in the toy bilinear problem  $\min_{x:|x|\leq1}\max_{y:|y|\leq1}xy$, the iterates still cycle/diverge, constrained only by the compact domain (imagine~\cref{fig:GDAxy} with a bounding box). This is a well-documented issue for applications like training GANs where convexity-concavity fails and average iterates are nonsensical models. 

\par These important issues have spurred a fruitful and extensive literature aimed at obtaining convergence without strong assumptions and without averaging iterates. Since the seminal paper~\citep{korpelevich1976extragradient}, the de facto approach has been to modify GDA with additional building blocks such as extragradients, momentum, optimism, anchoring, and more; see the prior work section below.

\subsection{Contribution: GDA actually can converge}\label{ssec:intro:cont}

Contrary to conventional wisdom, we show that GDA actually can converge in its original form~\eqref{eq:GDAupdate} by judiciously choosing stepsizes. The key innovation is the proposal of unconventional stepsize schedules---dubbed \emph{slingshot stepsize schedules}, described below---that are \emph{time-varying}, \emph{asymmetric}, and (most surprisingly) use interspersed \emph{negative stepsizes}. 
All three characteristics are provably necessary for GDA to converge even on the bilinear counterexample~\eqref{eq:xy}.

\par Our main results prove that these stepsize schedules enable GDA to converge on unconstrained bilinear, quadratic, and convex-concave min-max problems---the classical counterexamples for GDA.  Our results do not require extragradients, momentum, optimism, anchoring, or any other modification to GDA.

The focus of this paper is on whether GDA can converge at all and why our counterintuitive stepsizes are helpful, not on obtaining the fastest complexity bounds for any algorithm. Nevertheless, the resulting convergence rates are competitive with more involved algorithms. For bilinear and quadratic problems, our stepsize schedule enables GDA to converge at an accelerated rate which exactly matches the information-theoretically optimal rate among arbitrary first-order methods. For convex-concave problems, our stepsize schedule enables GDA to match the known convergence rates for extragradient and optimistic GDA. Finally, for strongly-convex-strongly-concave problems, GDA already converges with standard stepsize schedules (this is the only such setting in the paper) but at a slow rate; our stepsize schedule enables GDA to converge at the optimal accelerated rate for the first time. Our results do not require average iterates: they apply to the best iterate for the convex-concave setting and to the last iterate in all other settings, see the discussion in~\cref{ssec:intro:prior}.

\paragraph*{Key insight: progress in $2$ iterations.} The slingshot stepsize schedule is not designed to make progress in every individual GDA step, but rather over pairs of steps. This non-greedy viewpoint is essential: some individual steps use negative stepsizes and therefore move ``backward'', but the forward and backward movements can combine favorably over consecutive iterations. Thus the stepsizes must be designed and analyzed through their net effect over multiple iterations. 
\par The analysis of this two-step progress necessarily depends on the structure of $f$. When $\nabla f$ is linear, as for bilinear and quadratic objectives, the two-step dynamics can be characterized exactly: we show that paired slingshot steps reduce the dynamics to a polynomial iteration, which enables reducing this GDA stepsize design question to the classical question of choosing GD stepsizes for quadratic minimization via Chebyshev polynomials~\citep{young53}. See~\cref{ssec:plausibility:bilinear} for an overview. When $\nabla f$ is nonlinear, this polynomial connection breaks down, mirroring the classical distinction between quadratic and general convex minimization~\citep{d2021acceleration}. Nevertheless, the same two-step mechanism persists: a Taylor expansion shows that the forward/backward movements over two iterations cancel to first order and produce a second-order drift in a direction which enables convergence and is otherwise impossible for GDA to move in. See~\cref{ssec:plausiblity:approx} for an overview. To turn this local intuition into a global convergence result under only the standard assumptions of smoothness and (strong) convexity-concavity, we prove direct two-step progress guarantees. We develop these linear and nonlinear analyses in~\cref{sec:linear,sec:nonlinear}, respectively, and the Taylor-expansion interpretation in~\cref{sec:connections}. We provide these multiple derivations in order to give sharper convergence rates in specific settings, as well as to provide complementary interpretations of the two-step progress more generally.

\subsection{Slingshot stepsize schedules}\label{ssec:intro:slingshot}

We begin by introducing the proposed stepsize schedules.
Since optimal implementations vary across settings, we propose a \emph{family} of stepsize schedules. These schedules pair stepsizes in consecutive iterations $2t,2t+1$ so that:
\begin{itemize}
        \item[(i)] \underline{Negative stepsizes:} At least one of $\alpha_{2t},\alpha_{2t+1},\beta_{2t},\beta_{2t+1}$ is strictly negative. 
        
    \item[(ii)] \underline{Alternating products:} 
    $\alpha_{2t} \beta_{2t+1} \geq 0$ and $\alpha_{2t+1} \beta_{2t} \geq 0$.

    \item[(iii)] \underline{Consecutive sums:} 
    $\alpha_{2t} + \alpha_{2t+1} \geq 0$ and $\beta_{2t} + \beta_{2t+1} \geq 0$.
\end{itemize}
Property (iii) ensures that the forward movement outweighs the backward movement from (i), and (ii) ensures that the higher-order net movement arising from positive/negative cancellations is in a helpful direction. We describe this in detail below; for concreteness we first start with two examples.

\paragraph*{Concrete examples.} 
In this paper we focus on two instantiations of this general family of stepsize schedules, dichotomized by whether the min-max objective $f$ has a linear gradient $\nabla f$:
\begin{itemize}
	\item \underline{Linear $\nabla f$.} For such problems (i.e., bilinear and quadratic objectives), we propose slingshot stepsize schedules of the form
	\begin{align}\label{eq-intro:schedule-linear}
			\alpha_{2t} = -\beta_{2t} = -\alpha_{2t+1} = \beta_{2t+1} = h_t\
		\end{align}
		for a suitable choice of $\{h_t\}$. Observe that the minimization variable $x$ performs a positive step ($\alpha_{2t} > 0$) followed by a negative step ($\alpha_{2t+1} < 0$) if $h_t > 0$, and vice versa for the maximization variable $y$. Thus, in each iteration, the stepsizes for $x$ and $y$ are \emph{anti-correlated} in that they are equal and opposite. See~\cref{fig:stepsizes} for an illustration.

	\item \underline{Nonlinear $\nabla f$.} For such problems (e.g., convex-concave objectives), we propose slingshot stepsize schedules of the form
	\begin{align}\label{eq-intro:schedule-nonlinear}
		\alpha_{2t} = -\beta_{2t} = \beta_{2t+1} = h,\, \alpha_{2t+1} = 0 \qquad  \text{ or } \qquad -\alpha_{2t} =
		\beta_{2t} = \alpha_{2t+1} = h,\, \beta_{2t+1} = 0\,,
	\end{align}
	where we randomize this ``or''  with probability $1/2$ to break the symmetry between $x$ and $y$. Notice that properties (i)-(iii) hold almost surely, i.e., in every realization of the stepsizes. 
\end{itemize}
The reason we study multiple versions of the slingshot stepsize schedule is that the former proposal is exactly optimal among all stepsize schedules for settings where $\nabla f$ is linear (see~\cref{sec:linear}), but actually fails to converge in settings where $\nabla f$ is nonlinear (see~\cref{sec:nonlinear}). The latter proposal fixes this convergence issue for general convex-concave problems.

\begin{figure}
	\centering
	\includegraphics[width=0.5\linewidth]{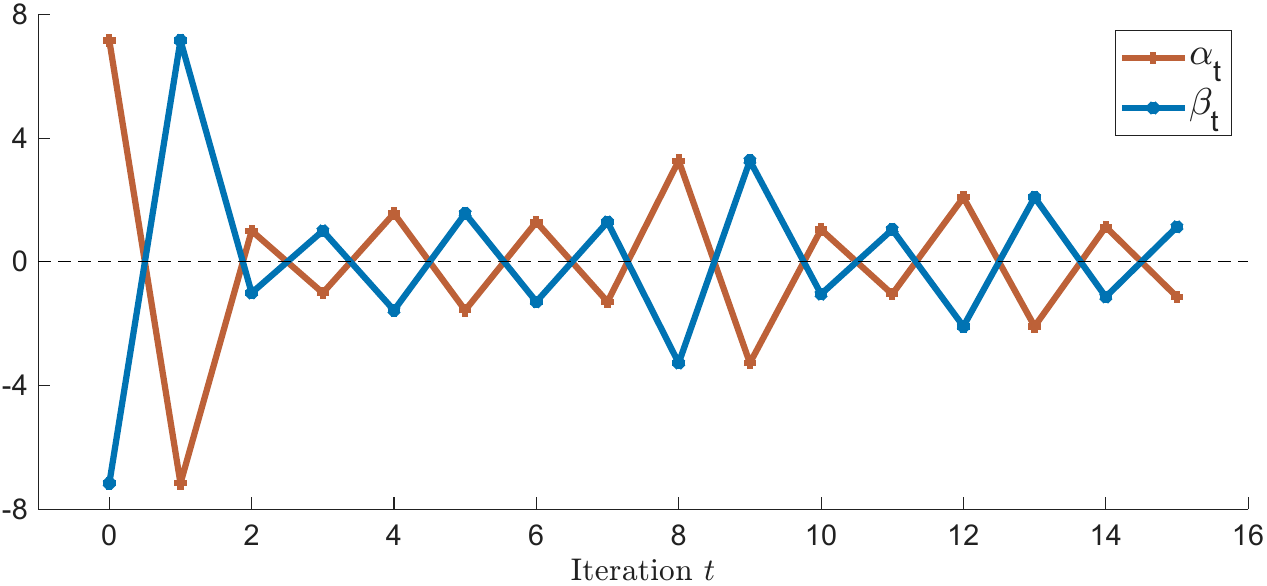} 
	\caption{\footnotesize Optimal stepsize schedule for unconstrained bilinear min-max optimization, here for $16$ iterations, for a bilinear coupling with spectral bounds $1/100 \leq \sigma_i^2(\nabla_{xy}^2 f) \leq 1$, and in the stable fractal ordering discussed in~\cref{app:stability}. Recall that $\alpha_t$ and $\beta_t$ respectively denote the stepsizes for $x$ and $y$; standard schedules take $\alpha_t,\beta_t \geq 0$ and often equal. Observe that the proposed stepsizes are time-varying, asymmetric, and alternately negative.}
	\label{fig:stepsizes}
\end{figure}

\paragraph*{Three key properties.}  Slingshot stepsize schedules are time-varying, asymmetric, and periodically use negative stepsizes. As already
mentioned, all three properties are provably necessary\footnote{
		This statement applies to \emph{real} stepsize schedules $\alpha_t,\beta_t \in \R$. Intriguingly, we show that \emph{complex} stepsize schedules can also make GDA converge (and this does not require asymmetry); details in~\cref{ssec:linear-alternative}. 
        This appears to be the first use of complex stepsizes in optimization or min-max optimization, and is surprising because the problem data is purely real.
        However, these complex stepsizes are specially tailored to the unconstrained bilinear setting and do not extend gracefully to more general settings, hence we focus on real stepsizes. \label{fn:complex}
} for GDA to converge even on bilinear problems. We comment on each of these three key properties:

\begin{itemize}
	\item 
	\underline{Negative stepsizes (i.e., $\alpha_t < 0$ or $\beta_t < 0$).} 
	Such steps move \emph{backward}: they move in the positive gradient direction (an ascent direction) for the minimization variable $x$, and in the negative gradient direction (a descent direction) for the maximization variable $y$. At first glance, it seems that such steps are nonsensical and should only worsen the standard divergent behavior of GDA. Counterintuitively, we show that such steps are essential for convergence. 
	Negative steps make backward progress, but de-synchronize the min and max variables (overcoming the cycling issue of GDA), and lead to a \emph{slingshot phenomenon} in which the forward progress in the other iterations (with positive stepsizes) is overwhelmingly larger. This results in fast overall convergence. The key intuition is that the positive and negative movements cancel to first order, resulting in a second-order movement in a new direction that leads to convergence and is otherwise impossible for GDA to move in. 
    We explain this in detail in~\cref{ssec:plausiblity:approx}.
	\item 
	\underline{Asymmetric stepsizes (i.e., $\alpha_t \neq \beta_t$).}
	Slingshot stepsize schedules update differently the minimization variable $x$ and maximization variable $y$. This goes beyond what is possible with classical reductions from min-max problems to monotone operator theory (which always concatenate the variables $z = (x,y)$ and then treat them symmetrically), highlighting the importance of the intrinsic asymmetry in min-max problems.
    \par In fact, fundamental obstructions prohibit extending the slingshot stepsizes from convex-concave min-max problems to variational inequalities with monotone operators. Taking different stepsizes in $x$ and $y$ is not even \emph{implementable} for generic monotone operator problems since operator monotonicity is invariant with respect to coordinate permutations and in particular does not have block structure $z = (x,y)$ as in min-max problems. Further, taking the same stepsizes in all variables---i.e., algorithms of the form $z_{t+1} = z_t - \eta_t F(z_t)$---fails to solve unconstrained variational inequalities $F(z) = 0$ for smooth monotone operators $F$ for any choice of stepsizes $\eta_t$; indeed, in the special case of min-max problems, such algorithms devolve into GDA with symmetric stepsizes, which fails to converge even for smooth bilinear problems (see~\cref{ssec:warmup:diverge}).  Our results exploit the block structure of the variables $z = (x,y)$ to implement the slingshot stepsizes, and exploit the convex-concave block structure of the monotone operator $F(z) = (\nabla_x f, -\nabla_y f)$ to analyze the slingshot stepsizes.
	\item
	\underline{Time-varying stepsizes (i.e., $\alpha_t,\beta_t$ are not constant in $t$).}
	Slingshot stepsize schedules do not make progress in each iteration (negative steps make backward progress). Nevertheless, we prove that this algorithm converges in the long-run. 
	\par We remark that slingshot stepsize schedules can be time-varying in two ways. Varying \emph{within} pairs of iterations (e.g., $\alpha_{2t} \neq \alpha_{2t+1}$ and $\beta_{2t} \neq \beta_{2t+1}$) enables GDA to converge. Further varying \emph{between} pairs of iterations (i.e., $h_t$ varying in $t$) can accelerate this convergence.
\end{itemize}
While some existing stepsize schedules have the latter two properties (diminishing stepsizes are time-varying, two-timescale GDA is asymmetric, and alternating GDA is time-varying and asymmetric), none have the first property. Recall that \emph{all} three properties are provably required for GDA to converge. Perhaps the closest similarity to negative stepsizes is negative momentum~\citep{gidel2019negative}, but that is fundamentally different as it changes GDA by adding internal dynamics, uses positive stepsizes, often has time-invariant, symmetric parameters, and yields trajectories that are qualitatively completely different (see~\cref{fig:diff-behavior}).
See the related work~\cref{ssec:intro:prior} for further discussion.

\subsection{Plausibility argument for bilinear problems}\label{ssec:plausibility:bilinear}

Here we provide a simple plausibility argument for why time-varying, asymmetric, and periodically negative stepsizes enable GDA to converge. For this argument, consider bilinear problems $\min_{x} \max_y x^{\top} \bm{B} y$ with stationary solution at the origin (without loss of generality by translation), and for simplicity assume that $\bm B$ has non-zero singular values with $m\bm{I} \preceq \bm{BB}^{\top}, \bm{B}^{\top}\bm{B} \preceq M\bm{I}$; the general case is a straightforward extension via SVD, see~\cref{ssec:bilinear}.

\paragraph*{A first convergence result.} 
For bilinear problems, two iterations of the GDA algorithm~\eqref{eq:GDAupdate} amounts to the linear dynamical system
\begin{align}
    \begin{bmatrix}
        x_2 \\ y_2
    \end{bmatrix} 
    &= \begin{bmatrix}
            \bm{I} & -\alpha_1\bm{B} \\\beta_1\bm{B}^{\top} & \bm{I}
    \end{bmatrix}
 \begin{bmatrix}
            \bm{I} & -\alpha_0\bm{B} \\\beta_0\bm{B}^{\top} & \bm{I}
    \end{bmatrix}
    \begin{bmatrix}
        x_0 \\ y_0
    \end{bmatrix}\,.
\end{align}
Consider using the slingshot stepsize schedule~\eqref{eq-intro:schedule-linear}, that is
\begin{align}\label{eq:warmup:2}
\alpha_0 = \beta_1 = h \qquad \text{ and } \qquad \beta_0 = \alpha_1 = -h\,,
\end{align}
for $h > 0$. Then this two-step update simplifies to
\begin{align}
    \begin{bmatrix}
        x_2 \\ y_2
    \end{bmatrix} 
    =
    \underbrace{\begin{bmatrix}
            \bm{I} & h \bm{B} \\ h\bm{B}^{\top} & \bm{I}
    \end{bmatrix}}_{\bm{U_1}}
    \underbrace{\begin{bmatrix}
            \bm{I} & -h\bm{B} \\ -h\bm{B}^{\top} & \bm{I}
    \end{bmatrix}}_{\bm{U_0}}
    \begin{bmatrix}
        x_0 \\ y_0
    \end{bmatrix}
    = \underbrace{\begin{bmatrix}
            \bm{I} - h^2 \bm{B}\bm{B}^{\top} & \bm{0} \\
            \bm{0} & \bm{I} - h^2 \bm{B}^{\top}\bm{B}
    \end{bmatrix}}_{\bm{U_1U_0}}
    \begin{bmatrix}
        x_0 \\ y_0
    \end{bmatrix}
    \,,
    \label{eq:warmup:U2}
\end{align}
which is contractive $\|\bm{U_1U_0}\| < 1$ for sufficiently small $h > 0$ since $\bm{B}\bm{B}^\top$ and $\bm{B}^{\top}\bm{B}$ are positive definite. It immediately follows that repeating this $2$-step schedule~\eqref{eq:warmup:2}---that is, $\alpha_n= h, \beta_n = -h$ for $n$ even, and vice versa for $n$ odd---makes GDA converge, and moreover at an exponential rate. 
            
            \par Notice that the key properties of the stepsizes~\eqref{eq:warmup:2} used for the two-step contraction~\eqref{eq:warmup:U2} are that $\alpha_0 + \alpha_1 = \beta_0 + \beta_1 = 0$ (so that the off-diagonal blocks of $\bm{U_1U_0}$ vanish) and $\alpha_0 \beta_1 = \alpha_1 \beta_0 > 0$ (so that the diagonal blocks of $\bm{U_1U_0}$ are contractive). These two properties correspond respectively to (iii) and (ii) in the definition of slingshot stepsize schedules in~\cref{ssec:intro:slingshot}.

			\paragraph*{Key point: contraction in two iterations, despite expansion in any single iteration.}
				To understand this result, it is important to emphasize that GDA converges (i.e., $\|\bm{U_1 U_0}\| <1$) even though each individual iteration is expansive in some directions (i.e., $\|\bm{U_1}\|, \| \bm{U_0}\| > 1$).\footnote{The expansiveness is because both $\bm{U_0}$ and $\bm{U_1}$ can have norm $1 + h \sqrt{M} > 1$; consider the eigenvectors $(1,1)$ and $(1,-1)$ and even the $1$-dimensional instances $\bm{B} = -\sqrt{M}$ or $\sqrt{M}$, respectively. 
				} The expansiveness of each individual step is due to the backward movement: the negative stepsizes $\alpha_1 < 0$ and $\beta_0 < 0$ respectively amount to an ascent step in the minimization variable $x$ and a descent step in the maximization variable $y$. Yet, overall convergence occurs because a variable's backward movement in one iteration is outweighed by the forward movement in the paired iteration. This is the ``slingshot'' behavior illustrated in~\cref{fig:GDAxy}.
				
				\par This can be intuitively understood via the following $2 \times 2$ example:
				\begin{align}
					\begin{bmatrix} 
						1 & \eps \\
						\eps & 1
					\end{bmatrix}
					\begin{bmatrix} 
						1 & -\eps \\
						-\eps & 1
					\end{bmatrix}
					=
					\begin{bmatrix}
						1 - \eps^2 & 0 \\
						0 & 1 - \eps^2
					\end{bmatrix}
                    \label{eq:warmup:2x2}
				\end{align}
				has norm smaller than $1$, even though both constituent matrices have norm larger than $1$. This occurs because the expansive eigenvectors $(1,1)$ and $(1,-1)$ for $\bm{U_1}$ and $\bm{U_0}$, respectively, are contractive eigenvectors for the opposite update matrix, and the contraction $1-\eps$ outweighs the expansion $1+\eps$ since $(1-\eps)(1+\eps) = 1-\eps^2 < 1$.
                
                \par In fact, this simple example~\eqref{eq:warmup:2x2} is the crux of why $\|\bm{U_1U_0}\| < 1$ in the two-step update~\eqref{eq:warmup:U2}: in $1$ dimension, $\bm{U_0}$ and $\bm{U_1}$ are of this form modulo permutation; and in general dimension, simultaneous diagonalization reduces $\bm{U_0}$ and $\bm{U_1}$ to this form on each eigenspace.

							\paragraph*{Optimizing the convergence rate.}
							
							So far we have shown that for sufficiently small $h > 0$, periodically repeating the two-step schedule~\eqref{eq:warmup:2} makes GDA converge at an exponential rate. This amounts to the slingshot stepsize schedule~\eqref{eq-intro:schedule-linear} with $h_t = h$ for all $t$. Now that convergence is ensured, it is natural to ask how fast can one make this exponential convergence? How small should $h$ be?  What is the optimal stepsize schedule? Should the optimal $h_t$ vary over time? 
							
							\par Conveniently, all of these questions are answered by appealing to the following simple but key interpretation of~\eqref{eq:warmup:U2}. We emphasize that this reduction applies only to the slingshot stepsize schedule and not to any other existing schedules.

							\begin{obs}
								[Paired GDA slingshot steps are equivalent to GD steps on the Hamiltonian]
								\label{obs:paired}
								Suppose $f$ is bilinear. Then the following are equivalent:
								\begin{itemize}
									\item Two iterations of GDA with the slingshot stepsize schedule~\eqref{eq:warmup:2} on $\min_x \max_y f(x,y)$. 
									\item One iteration of GD with stepsize $h^2$ on $\min_{x,y} 
								\frac{1}{2}\|\nabla f(x,y)\|^2$. 
								\end{itemize}
							\end{obs} 

                            Indeed, a simple calculation shows that one iteration of GD on the ``Hamiltonian'' $\Phi(x,y) := \tfrac{1}{2}\|\nabla f(x,y)\|^2 = \tfrac{1}{2}x^{\top} \bm{BB}^\top x + \tfrac{1}{2} y^{\top} \bm{B}^{\top}\bm{B} y$ is equivalent to the update~\eqref{eq:warmup:U2}. The upshot is that the Hamiltonian is a smooth, strongly-convex quadratic ($m \bm{I} \preceq \nabla^2 \Phi \preceq M \bm{I}$), hence we may appeal to classical analysis techniques and stepsize schedules for GD. 
							For example, simply repeating~\eqref{eq:warmup:U2} with $h = \sqrt{2/(M+m)}$ yields exponential convergence with contraction factor $(\kappa - 1)/(\kappa+ 1)$, where $\kappa := M/m$ denotes the condition number. Better yet, using optimal GD stepsize schedules~\citep{young53} improves this to the accelerated rate $(\sqrt{\kappa} - 1)/(\sqrt{\kappa} + 1)$. This is our proposed slingshot stepsize schedule~\eqref{eq-intro:schedule-linear} for bilinear problems, which uses time-varying $h_t$ based on the roots of Chebyshev polynomials. See~\cref{ssec:bilinear} for full details on this stepsize schedule and the fact that it not only converges, but moreover converges at the accelerated rate that is information-theoretically optimal among first-order methods.
                            See also~\cref{fig:conv} for an illustration of how the practical performance matches the theory.

							\begin{figure}
								\centering
								\includegraphics[width=0.75\linewidth]{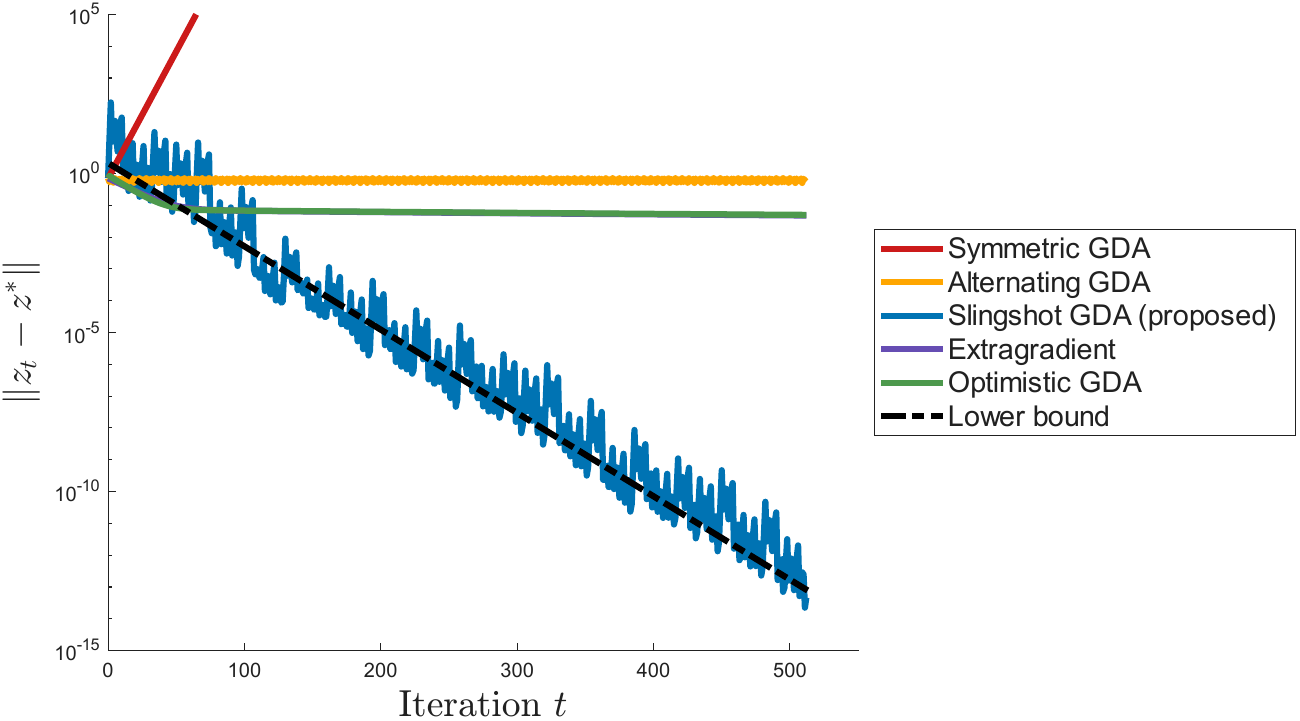}
								\caption{\footnotesize Comparison of algorithms on a random bilinear problem with condition number $\kappa = 300$. Using standard stepsize schedules, GDA diverges (red) or cycles (yellow). Known ways of making GDA converge require changing the algorithm  and do not always lead to fast convergence (purple, green---nearly overlapping here) due to suboptimal rates of $\kappa \log 1/\eps$ rather than $\sqrt{\kappa} \log 1/\eps$, see~\cref{tab:bilinear}. (The faster initial convergence is only because this is a random problem rather than worst-case.) The proposed slingshot stepsize schedule requires no changes to GDA beyond stepsizes and leads to the optimally accelerated convergence rate for bilinear problems. Performance on this random problem is similar to worst-case problems (dotted minimax lower bound); qualitatively similar behavior is observed for other problems. Further details in \cref{app:experimental-setup}.
								}
								\label{fig:conv}
							\end{figure}

    \paragraph*{Beyond bilinear problems.} The above calculations can be extended from bilinear to quadratic objectives, essentially since linearity of $\nabla f$ is the key property needed for the two-step argument~\eqref{eq:warmup:U2}; details in~\cref{sec:linear}. However, further analysis techniques are required beyond the setting of linear $\nabla f$, e.g., for convex-concave $f$. Of course one could try to linearize and then use a similar spectral analysis, but such arguments would be geared toward local rather than global convergence and moreover would necessitate assuming much stronger structural properties on $f$ than are necessary. In~\cref{sec:nonlinear} we prove progress after two steps for the standard setting where $f$ is only assumed convex-concave and smooth, and then use that to conclude global convergence.

	\subsection{Interpretation via second-order approximation}\label{ssec:plausiblity:approx}

	Here we provide a complementary interpretation of the two-step progress made by the slingshot stepsize schedule. This interpretation extends beyond the setting of linear $\nabla f$: it is exact if $\nabla f$ is linear, and holds approximately for the general setting where $\nabla f$ is nonlinear.

    \par The overarching idea is that positive/negative steps in consecutive iterations cancel to first order, leading to a second-order movement in a new direction that leads to convergence and is otherwise impossible for GDA to move in. For intuition, a useful analogy is \emph{parallel parking}. Parallel parking a car requires paired movements: first move forward (at an angle by turning the wheel), then backward by the same amount (at a slightly different angle by turning the wheel). The car's forward/backward movement cancel to first order, resulting in a second-order movement in a new direction (sideways) that leads to parking and is otherwise impossible for the car to move in. Geometrically, the alternating use of the two vector fields $\nabla f$ and $-\nabla f$ leads to a second-order drift (the new direction) which quantifies the \emph{non-reversibility} of gradient descent/ascent; see~\cref{sec:connections} for Lie bracket/commutator-type calculations and interpretations.
    This should be contrasted with standard GDA, which always follows the same vector field $(-\nabla_x f, \nabla_y f)$.

 \par The second-order movement can be explicitly computed by a Taylor expansion. For example, for the nonlinear version~\eqref{eq-intro:schedule-nonlinear} of the slingshot stepsize schedule, the expected movement from two steps of GDA is:
	\begin{align}
		\E 
		\begin{bmatrix}
			x_{2t+2} \\ y_{2t+2}
		\end{bmatrix}
		\approx
		\begin{bmatrix}
			x_{2t} \\ y_{2t}
		\end{bmatrix}
		+ 
		\underbrace{\frac{h}{2} 
			\begin{bmatrix}
				- \nabla_x f(x_{2t}, y_{2t}) \\
				\nabla_y f(x_{2t}, y_{2t})
		\end{bmatrix}}_{\text{\shortstack{standard direction: GDA update}}}
		- \underbrace{\frac{h^2}{2} \nabla^2 f(x_{2t}, y_{2t}) \nabla f(x_{2t}, y_{2t})}_{\text{\shortstack{new direction: HGD update \\ (from positive/negative cancellation)}}}
		\label{eq:css-2ndorder}
	\end{align}
	up to error terms of order $\mathcal{O}(h^3)$. See~\cref{sec:connections} for details. This expansion~\eqref{eq:css-2ndorder} can be interpreted as a combination of two movements: a standard ``forward'' GDA step and a new direction that is impossible in a single GDA step. Notice that $h/2$ in the first direction arises from $\E[\alpha_{2t} + \alpha_{2t+1}] = \E[\beta_{2t} + \beta_{2t+1}]$, and $h^2/2$ in the second direction arises from $\E[\alpha_{2t} \beta_{2t+1} + \alpha_{2t+1} \beta_{2t}]$; these are both positive due to properties (iii) and (ii) in the definition of slingshot stepsize schedules in~\cref{ssec:intro:slingshot}.

	\paragraph*{New direction: Hamiltonian gradient descent.} This new direction is a gradient descent update on $\min_{(x,y)} \Phi(x,y)$ where $\Phi(x,y) = \tfrac{1}{2}\|\nabla f(x,y)\|^2$ denotes the Hamiltonian. 
	This generalizes~\cref{obs:paired}, which only applies to bilinear objectives $f$. Three remarks are in order. First, this auxiliary problem $\min_{x,y} \Phi(x,y)$ is helpful for solving the original problem $\min_x \max_y f(x,y)$ since minimizers for the former are in correspondence with stationary solutions for the latter (assuming their existence). Second, note that GDA has access neither to gradients of $\Phi$ nor to Hessians of $f$; the paired cancellation in the slingshot stepsize schedule enables GDA to approximately move in such directions using only standard query access to $\nabla f$. Third, note that for this Hamiltonian descent direction, $x$ and $y$ \emph{coordinate} to jointly minimize the stationary criteria $\Phi(x,y)$.  This is in contrast to standard GDA, where $x$ and $y$ \emph{compete} to minimize and maximize $f$, respectively, which gives rise to the standard intuition for GDA's cycling behavior as illustrated in~\cref{fig:GDAxy}.

	\paragraph*{Effective two-step movement: consensus optimization.} The expansion~\eqref{eq:css-2ndorder} shows that two steps of GDA with the slingshot stepsize schedule approximately implements one step of consensus optimization, a popular algorithm for training GANs that involves both a standard ``forward'' GDA step and a step of Hamiltonian GD~\citep{Mescheder_Nowozin_Geiger_2017}. This algorithm's name arises from the interpretation that $x$ and $y$ are, in part, coordinating. An interesting difference is that consensus optimization exactly implements~\eqref{eq:css-2ndorder} for structured objectives $f$ using double backpropagation, whereas our algorithm approximately implements~\eqref{eq:css-2ndorder} via finite-differencing---which applies to general $f$ and requires only black-box queries to $\nabla f$. See~\cref{ssec:connections:nonlinear} for further discussion of similarities and differences.

	\paragraph*{From intuition to convergence.} This connection to consensus optimization provides some justification for our counterintuitive algorithm, since consensus optimization is well-documented to perform well empirically even in challenging non-convex landscapes like GANs~\citep{Mescheder_Nowozin_Geiger_2017}.
	However, this connection does not directly imply non-asymptotic\footnote{Asymptotic local convergence (without rates) for the slingshot stepsize schedule can be shown for non-convex-non-concave settings by taking arbitrarily small stepsizes---in which case the approximation~\eqref{eq:css-2ndorder} to consensus optimization becomes arbitrarily precise---and applying asymptotic local convergence guarantees for that algorithm, as in~\citep[Corollary 8]{Mescheder_Nowozin_Geiger_2017}.} convergence of the slingshot stepsize schedule. 
	This is both because the connection~\eqref{eq:css-2ndorder} is inexact, and more importantly, because the theoretical convergence properties of consensus optimization are still nascent---in particular although recent works have proved elegant convergence results~\citep{Mescheder_Nowozin_Geiger_2017,Abernethy_Lai_Wibisono_2021,Liang_Stokes_2019,Azizian_Mitliagkas_Lacoste-Julien_Gidel_2020}, these require strong assumptions and do not apply to standard nonlinear setups such as smooth convex-concave $f$ or even smooth strongly-convex-strongly-concave $f$. Therefore, when we prove our convergence results for the setting of smooth convex-concave $f$ in~\cref{sec:nonlinear}, rather than building upon this connection to consensus optimization, we instead directly prove progress for two steps of the slingshot stepsize schedule. (In fact, in the other direction, we are hopeful that the convergence results we show for the slingshot stepsizes may be adaptable to prove analogous convergence results for consensus optimization, which would help justify the strong empirical performance of that algorithm over the past decade.)

\input{sections/related_work}

%% file: sections/related_work.tex
\subsection{Related work}\label{ssec:intro:prior}

There is an extensive literature on GDA since it and its variants form the primary workhorse for solving large-scale min-max problems.
In addition to the works already cited above, here we contextualize further with the closest lines of work. See~\cref{tab:bilinear,tab:cc,tab:scsc} for complexities of variants of GDA for the various standard settings studied in this paper.

\paragraph*{Performance criteria.} One reason that there is an extensive literature is that there are different ways of measuring the error of an approximate min-max solution $z = (x,y)$. Two standard criteria are approximate stationarity $\|\nabla f(z)\|$ and distance to the solution set $\inf_{\textrm{solution }z^*} \|z - z^*\|$; our results apply to the former criterion in all settings, and also apply to the second criterion in all settings where such a result is possible\footnote{For example, there exist unconstrained convex-concave min-max problems (such as arbitrarily flat quadratic functions) for which it is impossible for any algorithm to achieve non-trivial rates for minimizing the distance to a solution. Details in~\cref{rem:quad-metrics}.}. 
Another criterion is the duality gap $\max_{y'} f(x,y') - \min_{x'} f(x', y)$; however that is typically studied in constrained settings as it is potentially infinite in unconstrained settings (e.g., $\min_{x \in \R} \max_{y \in \R} xy$). 
The literature has also studied alternative criteria including variations of approximate stationarity specifically for constrained settings~\citep{Lin_Jin_Jordan_2020,Lin_Jin_Jordan_2025} and restricted duality gaps for unconstrained settings~\citep{Nesterov_2007, Mokhtari_Ozdaglar_Pattathil_2020, Jiang_Mokhtari_2024}, and it would be interesting to investigate our stepsizes in these contexts as well. 

\paragraph*{Best/averaged/last iterate.} Another reason that there is an extensive literature is that the performance criterion can be measured at either the best iterate $\argmin_{z \in \{z_1, \dots, z_T\}} \| \nabla f(z_t)\|$, the averaged (a.k.a.\,ergodic) iterate $\tfrac{1}{T} \sum_{t=1}^T z_t$, or the last iterate $z_T$. 
Early convergence rates were primarily for the best and averaged iterates since this can simplify the analyses, and recent works have also shown last iterate results, see for example \citep{gorbunov2022extragradient,Gorbunov_Taylor_Gidel, cai2022tight, Daskalakis_Ilyas_Syrgkanis_Zeng_2018,Abernethy_Lai_Wibisono_2021} and the references within. 
Note that both best and last iterate results do not require averaging the iterates, an important benefit for practical settings such as training GANs (see the discussions in e.g.,~\citep{hsieh2021limits,mertikopoulos2019optimistic}). Our results apply to the best iterate for the convex-concave setting and to the last iterate in all other settings. It is a natural question to extend our convex-concave result to the last iterate; this seems potentially plausible given that the analogous results hold for extragradient and optimistic GDA~\citep{gorbunov2022extragradient,Gorbunov_Taylor_Gidel,cai2022tight}.

\paragraph{GDA with different stepsize schedules.} As described above (see~\cref{fig:GDAxy} for an illustration and~\cref{ssec:warmup:diverge} for short proofs), all previous versions of GDA fail to converge on simple problems like $\min_x \max_y xy$. However, two popular stepsize schedules have been proposed which can enable convergence in some new settings (although still not problems like $\min_x \max_y xy$), accelerate convergence in settings where GDA is already known to converge, and also help in practice.

\emph{Alternating GDA} updates one variable at a time by 
selecting $\alpha_{2t} > 0, \beta_{2t} = 0$ (descent step in $x$) and $\alpha_{2t+1} = 0, \beta_{2t+1} > 0$ (ascent step in $y$) in consecutive iterations. This alternation improves divergence to limit cycling for the problem $\min_x \max_y xy$ (see~\cref{fig:GDAxy}) and can help empirically in some non-convex-non-concave settings such as training GANs \citep{Lu_Singh_Chen_Chen_Hong_2019}. Alternating GDA can also lead to faster convergence in settings where GDA already converges, namely smooth strongly-convex-strongly-concave problems. Specifically,~\citep{Lee_Cho_Yun_2024} shows that alternating GDA enables taking larger stepsizes, leading to a rate of $\mathcal{O}(\kappa^{1.5}\log1/\epsilon)$, which is faster than the $\Theta(\kappa^2 \log 1/\epsilon)$ fastest possible rate using symmetric stepsizes $\alpha_t = \beta_t$.~\citep{Zhang_Wang_Lessard_Grosse_2022} shows that alternating GDA achieves a $\mathcal{O}(\kappa\log1/\epsilon)$ local convergence rate under the additional assumptions of arbitrarily close initialization and continuous Hessian near optimality. This rate $\mathcal{O}(\kappa\log1/\epsilon)$ is the optimal rate among first-order methods~\citep{Ibrahim_Azizian_Gidel_Mitliagkas_2020}, and our slingshot stepsizes achieve this rate globally and without any additional such assumptions (see~\cref{tab:scsc}), and moreover enable convergence in further settings (see~\cref{tab:bilinear,tab:cc}). 

\emph{Two-timescale GDA} uses constant but asymmetric stepsizes (i.e., $\alpha\neq \beta$). This algorithm was proposed for training Generative Adversarial Networks (GANs) where the discriminator and generator (a.k.a.\ the $x$ and $y$ variables) train at different speeds in practice \citep{Heusel_Ramsauer_Unterthiner_Nessler_Hochreiter_2017}, and is popular in settings with different structural assumptions on the minimization objective $f(\cdot,y)$ and the maximization objective $f(x,\cdot)$. 
An influential result of \citep{Lin_Jin_Jordan_2020} shows that for certain non-convex-concave settings and certain notions of convergence, two-timescale GDA can converge, namely if $f(\cdot,y)$ is Lipschitz (but potentially non-convex), $f(x,\cdot)$ is concave, $f$ is smooth, and the domain of $y$ is compact. Note that changing the notion of convergence is essential since the iterates of two-timescale GDA diverge exponentially fast from the solution of unconstrained bilinear problems due to the time-invariance and positivity of the stepsize schedule, see~\cref{ssec:warmup:diverge} for a proof. 
Refined convergence rates for two-timescale GDA in this non-convex-concave setting are an active area of research, see e.g.,~\citep{Li_Farnia_Das_Jadbabaie_2022, Lin_Jin_Jordan_2025} and the references within. More broadly beyond standard GDA, two-timescale dynamics have been useful in a variety of settings including stochastic optimization \citep{Borkar_1997,Lin_Jin_Jordan_2020}, mean field games \citep{Lu_2023}, extragradient-type algorithms \citep{Chae_Kim_Kim_2023}, and more.

 \paragraph{Negative parameters in optimization.} Momentum is a powerful tool in convex optimization~\citep{polyak1964some,nesterov1983convex}, but the direct use of (positive) momentum fails to help GDA converge on simple problems like $\min_x \max_y xy$. An interesting idea of~\citep{gidel2019negative} is to use \emph{negative momentum}: this can be helpful empirically for training GANs, has interpretations via optimism~\citep{mokhtari2020unified}, and leads to convergence in bilinear~\citep{gidel2019negative} and strongly-convex-strongly-concave settings~\citep{Zhang_Grosse}, albeit at suboptimal convergence rates~\citep{Zhang_Wang_2021}, see~\cref{tab:bilinear,tab:scsc}. Negative momentum bears some intuitive similarities to our slingshot stepsizes in that both involve some sort of backward movement, and is therefore perhaps the closest existing algorithm to ours. However, negative momentum is fundamentally different for multiple reasons: it changes GDA by adding internal dynamics, it has only been used with positive stepsizes, and its use with symmetric GDA (as in~\citep{Zhang_Grosse,Zhang_Wang_2021}) makes the algorithm symmetric and time-invariant. Even visually, the trajectories of negative momentum and negative stepsizes are qualitatively completely different (see~\cref{fig:diff-behavior}).\footnote{Moreover, until very recently, it was an open question whether negative momentum could converge generally in convex-concave settings or could accelerate in any structured settings where convergence was already known; in contrast, this paper shows that all of this is possible using slingshot stepsizes. We resolved these questions in~\citep{shugart2026negative}, but the underlying mechanism appears to be quite different than for the slingshot stepsizes.}

In standard (non-min-max) optimization, negative parameters have barely been investigated. The only references we are aware of are \citep{Dai_1999, Shea_Schmidt_2024}, who allow for (but do not necessarily use) both positive and negative stepsizes in their line search in order to exploit negative curvature directions in non-convex problems. This usage of negative stepsizes is for an entirely different reason than in the present paper, and moreover we note that their line search would \emph{not} use negative stepsizes for convex settings.

\paragraph{First-order methods in min-max optimization.}
As mentioned above, although the stepsize schedules in alternating GDA and two-timescale GDA can be helpful in certain ways, they do not fix the non-convergence issues of GDA in simple settings like unconstrained bilinear problems. Therefore much of the literature on first-order algorithms for min-max optimization has focused on modifying GDA beyond stepsizes to enable convergence.

The first proposed algorithm for solving general convex-concave min-max problems was the \emph{extragradient method} (EGD) \citep{korpelevich1976extragradient}. In each iteration (from $z_t$ to $z_{t+1}$), EGD performs a step of GDA from the current iterate $z_t$ to create an auxiliary iterate $z_t'$ (look-ahead step), and then uses the gradient at $z_t'$ as the update direction from $z_t$ (movement step). Asymptotic convergence for EGD was first shown by~\citep{korpelevich1976extragradient}, and substantial work since has focused on non-asymptotic rates. For convex-concave objectives, \citep{Solodov_Svaiter_1999} proved a $\mathcal{O}(1/\epsilon)$ rate for best-iterate convergence in squared gradient norm; a matching last-iterate rate was first shown under the additional assumption of Lipschitz Hessians~\citep{Golowich_Pattathil_Daskalakis_Ozdaglar_2020} and was later strengthened to only require standard first-order smoothness assumptions~\citep{gorbunov2022extragradient}. Matching $\mathcal{O}(1/\epsilon)$ rates have also been shown for the ergodic iterate for the (restricted) gap function~\citep{Nemirovski_2004,Golowich_Pattathil_Daskalakis_Ozdaglar_2020,Mokhtari_Ozdaglar_Pattathil_2020}. For strongly-convex-strongly-concave problems,~\citep{TSENG1995237} showed an optimal $\mathcal{O}(\kappa\log1/\epsilon)$ convergence rate for EGD. Our results match these $\mathcal{O}(1/\epsilon)$ and $\mathcal{O}(\kappa\log1/\epsilon)$ rates (for the best and last iterates, respectively) by using GDA in its original form, see~\cref{tab:cc,tab:scsc}.

A related algorithm is the \emph{past-extragradient method}, nowadays commonly known as~\emph{optimistic GDA} (OGDA) \citep{Popov_1980}. OGDA performs a look-ahead step, similar to EGD, however this look-ahead step is computed differently by using prior gradient information. OGDA has several interpretations and has been studied through the lens of no-regret algorithms~\citep{Rakhlin_Sridharan_2013a, Rakhlin_Sridharan_2013b}, as an approximation of the proximal point method \citep{Jiang_Mokhtari_2024}, and as a variation on Nesterov-style momentum for convex optimization where ``momentum'' arises from the difference of consecutive gradients rather than the difference of consecutive iterates~\citep{mokhtari2020unified}. Last-iterate convergence results for squared gradient norm have been shown for OGDA, achieving the optimal rate $\mathcal{O}(\kappa \log1/\epsilon)$ for strongly-convex-strongly-concave problems \citep{mokhtari2020unified} and $\mathcal{O}(1/\epsilon)$ for smooth convex-concave problems \citep{ Gorbunov_Taylor_Gidel, cai2022tight}. Our results match these $\mathcal{O}(1/\epsilon)$ and $\mathcal{O}(\kappa\log1/\epsilon)$ rates (for the best and last iterates, respectively) by using GDA in its original form, see~\cref{tab:cc,tab:scsc}.

Recent work achieves accelerated convergence rates in the smooth convex-concave setting, improving $\mathcal{O}(1/\eps)$ to the optimal rate $\mathcal{O}(1/\sqrt{\epsilon})$ for squared gradient norm~\citep{yoon2021accelerated,Lee_Kim_2021,Tran_Dinh_Luo_2021,Yoon_Kim_Suh_Ryu_2024}. These accelerated methods primarily rely on combining either EGD or OGDA with anchoring, a mechanism similar to Halpern iterations for root finding, where after each update, the current iterate is moved slightly towards the initial iterate \citep{yoon2021accelerated,Lee_Kim_2021,Tran_Dinh_Luo_2021}. Alternative accelerated algorithms can also be derived by applying a certain ``H-duality'' operation to those algorithms~\citep{Yoon_Kim_Suh_Ryu_2024}.

\clearpage

\begin{table}[!t]
\footnotesize
\centering
\caption{\footnotesize Complexity of GDA variants for unconstrained \textbf{\emph{bilinear}} problems $\min_x \max_y x^{\top} \bm{B} y$. Here $\kappa = \sigma_{\max}^2(\bm B) / \sigma_{\min}^2(\bm B)$ where $\sigma_{\min}$ denotes the smallest non-zero singular value, and $\eps$ denotes the target accuracy for distance to a solution (rates are identical for $\|\nabla f\|^2$ up to a logarithm). Both~\citep{azizian2020accelerating} and~\cref{thm:bilinear-ub} exactly achieve the accelerated rate $\mathcal{O}(\sqrt{\kappa}\log 1/\epsilon)$ which is information-theoretically optimal among arbitrary first-order methods.}
\begin{tabular}{c|c|c|c}
& \textbf{Algorithm} & \textbf{Rate} & \textbf{Reference} \\ \Xhline{3\arrayrulewidth}
\multirow{3}{*}{GDA} 
    & Simultaneous GDA & No convergence & Folklore, e.g.,~\citep{korpelevich1976extragradient} \\ 
    & Alternating GDA & No convergence & \citep{gidel2019negative} \\ 
    & \textbf{Slingshot GDA}   & $\boldsymbol{\mathcal{O}(\sqrt{\kappa}\log(1/\epsilon))}$ & \cref{thm:bilinear-ub} \\ \hline
\multirow{4}{*}{Additional dynamics} 
    & Negative Momentum & $\mathcal{O}(\kappa\log(1/\epsilon))$ & \citep{gidel2019negative} \\
    & Optimistic GDA        & $\mathcal{O}(\kappa\log(1/\epsilon))$ & \citep{Daskalakis_Ilyas_Syrgkanis_Zeng_2018} \\ 
    & Extragradient         & $\mathcal{O}(\kappa\log(1/\epsilon))$ & \citep{azizian2020accelerating} \\
    & Heavy Ball on Hamiltonian & $\mathcal{O}(\sqrt{\kappa}\log(1/\epsilon))$ & \citep{azizian2020accelerating} 
\end{tabular}

\label{tab:bilinear}
\end{table}

\begin{table}[!t]
\footnotesize
\centering
\caption{\footnotesize Complexity of GDA variants for $\|\nabla f\|^2 \leq \epsilon$ for unconstrained smooth \textbf{\textit{convex-concave}} problems. Here the smoothness $L=1$; multiply all results by $L^2$ for general $L$.  A lower bound of $\Omega(1/\sqrt{\epsilon})$ is shown in~\citep{yoon2021accelerated}. Here our result applies to the best iterate, see the discussion around \cref{thm:gd-cc}.
}
\begin{tabular}{c|c|c|c}
 & \textbf{Algorithm} & \textbf{Rate} & \textbf{Reference} \\ \Xhline{3\arrayrulewidth}
\multirow{3}{*}{GDA} 
    & \multicolumn{1}{c|}{Simultaneous GDA} & \multicolumn{1}{c|}{No convergence} & Folklore, e.g.,~\citep{korpelevich1976extragradient} \\ 
    & \multicolumn{1}{c|}{Alternating GDA} & \multicolumn{1}{c|}{No convergence} & \citep{gidel2019negative} \\ 
    & \multicolumn{1}{c|}{\textbf{Slingshot GDA}}   & \multicolumn{1}{c|}{$\boldsymbol{\mathcal{O}(1/\epsilon)}$} & \cref{thm:gd-cc} \\ \hline
\multirow{3}{*}{Additional dynamics} 
    & \multicolumn{1}{c|}{Extragradient}         & \multicolumn{1}{c|}{$\mathcal{O}(1/\epsilon)$} 
        & \citep{Solodov_Svaiter_1999,gorbunov2022extragradient} \\ 
    & \multicolumn{1}{c|}{Optimistic GDA}          & \multicolumn{1}{c|}{$\mathcal{O}(1/\epsilon)$} 
        & \citep{Gorbunov_Taylor_Gidel,cai2022tight} \\ 
   & \multicolumn{1}{c|}{Extra anchored gradient/variants} & \multicolumn{1}{c|}{$\mathcal{O}(1/\sqrt{\epsilon})$} 
        & \citep{yoon2021accelerated} \\ 
\end{tabular}

\label{tab:cc}
\end{table}

\begin{table}[!t]
\footnotesize
\centering
\caption{\footnotesize Complexity of GDA variants for unconstrained $L$-smooth $\mu$-\textbf{\textit{strongly-convex-strongly-concave}} problems. Here $\kappa := L/\mu$
and $\eps$ denotes the target accuracy for distance to optimum (rates are identical for $\|\nabla f\|^2$ up to a logarithm). This is the only setting in the paper for which GDA converges with standard stepsizes; the slingshot stepsizes accelerate convergence to match the information-theoretically optimal rate of $\Omega(\kappa \log 1/\eps)$ from~\citep{Ibrahim_Azizian_Gidel_Mitliagkas_2020, Zhang_Hong_Zhang_2022}.}
\begin{tabular}{c|c|c|c}
 & \textbf{Algorithm} & \textbf{Rate} & \textbf{Reference} \\ \Xhline{3\arrayrulewidth}
\multirow{3}{*}{GDA} 
    & Simultaneous GDA  & $\mathcal{O}(\kappa^2\log 1/\epsilon)$     & \citep{Liang_Stokes_2019} \\ 
    & Alternating GDA   & $\mathcal{O}(\kappa^{1.5}\log 1/\epsilon)$ & \citep{Lee_Cho_Yun_2024} \\ 
    & \textbf{Slingshot GDA}     & $\boldsymbol{\mathcal{O}(\kappa\log 1/\epsilon)}$       & \cref{thm:gd-scsc} \\ \hline
\multirow{3}{*}{Additional dynamics} 
    & Negative Momentum & $\mathcal{O}(\kappa^{1.5}\log 1/\epsilon)$ & \citep{Zhang_Grosse} \\ 
    & Extragradient     & $\mathcal{O}(\kappa\log 1/\epsilon)$       & \citep{TSENG1995237} \\ 
    & Optimistic GDA    & $\mathcal{O}(\kappa\log 1/\epsilon)$       & \citep{mokhtari2020unified} \\ 
\end{tabular}

\label{tab:scsc}
\end{table}

\clearpage

\paragraph{Performance estimation.}
Many recent advances in the design and analysis of first-order min-max algorithms have been guided by computer-assisted analyses. This line of work is based on \textit{performance estimation problems} (PEP), pioneered by \citep{drori2014performance}. The overarching idea is that the search for a proof of an algorithm's convergence can be recast as a certain semidefinite program. These techniques were originally proposed for convex optimization and recently have been adapted to min-max problems in order to design new algorithms and prove tighter convergence rates for existing algorithms~\citep{Ryu_Taylor_Bergeling_Giselsson_2020, Zhang_Grosse, Yoon_Kim_Suh_Ryu_2024, gorbunov2022extragradient, Gorbunov_Taylor_Gidel, cai2022tight, Zhang_Wang_Lessard_Grosse_2022, Lee_Cho_Yun_2024}. Our analysis of bilinear and quadratic problems uses only simpler spectral-type arguments, but our analysis for convex-concave problems is inspired by PEP techniques, see~\cref{ssec:2step} for a discussion of the similarities and differences.

\paragraph{Min-max problems with structure.}
Specialized techniques can enable faster convergence rates for problems with additional structure. For example,~\citep{azizian2020accelerating} proposed a specialized algorithm specifically for bilinear problems that achieves the optimal accelerated rate. Interestingly, our slingshot stepsizes enable matching this optimal rate while using GDA in its original form, see~\cref{tab:bilinear}.
Faster convergence has also been shown under other structural assumptions, for example smooth strongly-convex-strongly-concave problems where the smoothness and strong convexity/concavity parameters are different for $x$ and $y$~\citep{Zhang_Hong_Zhang_2022, Lin_Jin_Jordan_2020_opt, Jin_Sidford_Tian_2022, Borodich_Kormakov_Kovalev_Beznosikov_Gasnikov_2023}, and bilinearly coupled problems of the form $f(x,y) = f_x(x) - f_y(y) + x^{\top} \bm B y$~\citep{Chambolle_Pock_2011,Gabay_Mercier_1976,Glowinski_Marroco_1975,Kovalev_Borodich_2024,boyd2011distributed}, and it is an interesting question if the slingshot stepsizes can similarly lead to improved rates under such assumptions, see the discussion in~\cref{sec:discussion}.

\paragraph{Stepsize optimization in minimization problems.} An orthogonal line of work aims to accelerate gradient descent for (non-min-max) convex optimization by judiciously choosing stepsizes. Classically, in 1953, Young showed that for the special case of convex quadratics, using stepsizes based on the roots of Chebyshev polynomials enables accelerating GD~\citep{young53}. Starting with~\citep{altschuler2018greed}, a recent line of work has shown that stepsize-based acceleration is possible for convex optimization beyond quadratics, see~\citep{altparrilo26survey} for a recent overview of this extremely active area of research. 
For the purpose of the present paper, the relevance is that all of those results are for (non-min-max) convex optimization; the present paper shows for the first time that these types of ideas can similarly accelerate GDA in min-max optimization. Specifically, in~\cref{sec:linear} we show how the positive/negative cancellation in the slingshot stepsize schedule reduces accelerating GDA for min-max problems with linear $\nabla f$, to accelerating GD for convex quadratic minimization. This leads to analogously accelerated rates using stepsizes based on roots of Chebyshev polynomials. It seems potentially plausible that in convex-concave settings, the slingshot stepsize schedule could also be accelerated by combining it with ideas from the silver stepsize schedule~\citep{AP23a,AP23b}, which enables accelerating GD to the rate $\mathcal{O}(\eps^{-\log_{1+\sqrt{2}} 2}) \approx \mathcal{O}(\eps^{-0.7864})$ for (non-quadratic) smooth convex optimization.

\paragraph*{Negative and complex stepsizes in different areas of mathematics.} Classically, negative and complex stepsizes arise in two different areas of mathematics for entirely different purposes: complex stepsizes arise in numerical differentiation~\citep{squire1998using,lyness1967numerical,martins2003complex} (in order to improve numerical stability), and negative or complex stepsizes arise in splitting methods for simulating differential equations (in order to achieve methods of order greater than $2$), see for example the survey~\citep{Blanes_Casas_Murua_2024} and the references within. Although these stepsizes arise for entirely different reasons than in our setting of min-max optimization, a common theme across these problems is that the algorithms are obtained by choosing parameters (e.g., stepsizes) so that low-order coefficients vanish in an associated Taylor expansion. This gives rise to a system of polynomial equations, one per coefficient.  Negative or complex stepsizes can then arise when the solutions require negative or complex roots.

%% file: sections/prelim.tex
\section{Preliminaries}\label{sec:prelim}

\subsection{Min-max optimization}\label{ssec:prelim:min-max}

In this paper we consider unconstrained min-max optimization problems, i.e., problems of the form
\begin{align}\label{eq:prelim}
	\min_{x \in \R^{d_x}} \max_{y \in \R^{d_y}} f(x,y)\,.
\end{align}
For simplicity of exposition, we consider finite-dimensional spaces $\R^{d_x}$ and $\R^{d_y}$, although our results extend to infinite-dimensional Hilbert space $\ell_2$. We focus on the fundamental setting where $f$ is $L$-smooth (i.e., $\nabla f$ is $L$-Lipschitz) and convex-concave (i.e., $f(\cdot,y)$ is convex for any $y$, and $f(x,\cdot)$ is concave for any $x$). 
In the paper we illustrate the slingshot stepsize schedule for several settings (bilinear $f$, quadratic $f$, convex-concave $f$, and strongly-convex-strongly-concave $f$), and for the convenience of the reader, we recall these standard definitions in the corresponding sections.

\paragraph*{Solutions and approximate solutions.} A saddle point of~\eqref{eq:prelim} is a point $z^* = (x^*,y^*)$ satisfying
\begin{align*}
	f(x^*,y) \leq f(x^*,y^*) \leq f(x,y^*)\,, \qquad \forall x,y\,.
\end{align*}
Intuitively, such a point is an equilibrium in the sense that the ``$x$-player'' cannot improve the objective by deviating from the strategy $x^*$ if the $y$-player does not deviate from $y^*$, and vice versa.

\par A related solution concept is a stationary point, which is a point $z^* = (x^*,y^*)$ satisfying
\begin{align*}
	\| \nabla f(z^*) \| = 0\,.
\end{align*}
In general, the saddle-point condition is more stringent than the stationary condition. However, for unconstrained convex-concave problems, the two notions are equivalent and can be respectively interpreted as global and local characterizations of optimality. Hence we use these terms interchangeably. Throughout we assume that such a solution point exists; this precludes degenerate min-max problems such as $f(x,y) = x$ for which the value of~\eqref{eq:prelim} is not finite and never achieved.

\par When numerically computing solutions, as is the goal in this paper, one can only hope to find approximate solutions $z = (x,y)$. Standard notions of approximate solutions are proximity to a stationary point $\|z - z^*\| \leq \eps$ and approximate stationarity $\|\nabla f(z)\| \leq \eps$. Note that the results of the former type imply results of the latter type under smoothness (since $\|\nabla f(z)\| = \| \nabla f(z) -  \nabla f(z^*)\| \leq L \|z- z^*\|$), however the latter is possible in settings where the former is not (e.g., see~\cref{rem:quad-metrics}).

\subsection{Failure of GDA with standard stepsizes}\label{ssec:warmup:diverge}
We motivate our algorithm by first recalling the folklore explanation for why $\GDA$ does not converge. Consider even the simple bilinear min-max problem
\begin{align}
	\min_{x \in \R} \max_{y \in \R} \, xy \,.
	\label{eq:diverge-xy}
\end{align}
The continuous-time GDA flow 
\begin{align*}
    \begin{bmatrix}
        \dot{x}(t) \\ \dot{y}(t)
    \end{bmatrix}
    =
    \begin{bmatrix}
        -\nabla_x f(x(t),y(t)) \\ \nabla_y f(x(t),y(t))
    \end{bmatrix}
    =
    \begin{bmatrix}
        - y(t) \\ x(t)
    \end{bmatrix}
\end{align*}
fails to converge since $x^2 + y^2$ is a conserved quantity. This results in cycling in circles around the unique saddle-point solution $(x^*,y^*) = (0,0)$. See the flow lines in~\cref{fig:GDAxy}.

\par This failure extends to discretizations of these continuous-time dynamics. Specifically, the following lemma shows that for standard types of stepsizes, GDA fails to converge in the sense that its iterates remain bounded away from the unique saddle-point solution $(x^*,y^*) = (0,0)$. This means that GDA also fails to converge in the sense of approximate stationarity.

\begin{lemma}[Failure of GDA with standard stepsizes]\label{lem:diverge-toy}
	GDA does not converge on the min-max problem~\eqref{eq:diverge-xy} for stepsize schedules $\alpha_t, \beta_t \in \R$ that fall under any of the following categories:
	\begin{itemize}
		\item [(i)] \underline{Constant stepsizes:} $\alpha_t \equiv \alpha$ and $\beta_t \equiv \beta$ for all iterations $t$, for some $\alpha,\beta$.
		\item [(ii)] \underline{Non-negative stepsizes:} $\alpha_t, \beta_t \geq 0$ for all iterations $t$.
		\item [(iii)] \underline{Symmetric stepsizes:} $\alpha_t = \beta_t$ for all iterations $t$.
	\end{itemize}
\end{lemma}
\begin{proof}
	For this problem, the dynamics of $\GDA$ simplify to the linear dynamical system
	\begin{align}
		\begin{bmatrix}x_{t+1}\\ y_{t+1}\end{bmatrix} = \bm{U_t} \begin{bmatrix}x_{t}\\ y_{t}\end{bmatrix} \qquad \text{ where } \qquad \bm{U_t} = \begin{bmatrix} 1 & -\alpha_t \\ \beta_t & 1\end{bmatrix}\,.
	\end{align}
	\par (i) For constant stepsizes, $\bm{U_t }\equiv \bm U$ for all $t$. It is a basic fact that such a time-invariant linear system converges from arbitrary initialization if and only if the spectral radius $\rho(\bm U) < 1$. But $\rho(\bm U) \geq 1$ since the eigenvalues of $\bm U$ are $1 \plusminus i \sqrt{\alpha \beta}$ and $\alpha,\beta \in \R$. Hence $\GDA$ is not convergent.
	\par (ii) For non-negative stepsizes, $\det(\bm U_t)= 1+\alpha_t\beta_t\geq 1$. Since the spectral norm is lower bounded by the spectral radius, which is in turn lower bounded by the geometric mean of the eigenvalues, we have $\|\prod_{t=1}^T \bm U_t \| \geq \rho(\prod_{t=1}^T \bm{U_t}) \geq \sqrt{\det(\prod_{t=1}^T \bm{U_t})} = \prod_{t=1}^T \sqrt{ \det( \bm{U_t})}  \geq 1$. Hence $\GDA$ is not convergent. 
	\par (iii) For symmetric stepsizes, $\det(\bm U_t)= 1+\alpha_t^2\geq 1$. Then argue as in (ii).
\end{proof}

In fact, the failure of $\GDA$ with these types of stepsize schedules is \emph{universal}: $\GDA$ diverges with stepsizes that are constant, non-negative, and/or symmetric not only for the toy problem~\eqref{eq:diverge-xy}, but in fact for \emph{any} bilinear min-max game. The proof essentially just follows by diagonalizing: $\GDA$ fails on each coordinate of a general bilinear min-max problem for the same reason that it fails on the univariate example above; details in~\cref{app:ub-diverge}. 

\begin{lemma}[Universal failure of standard $\GDA$]\label{lem:diverge-universal}
	\cref{lem:diverge-toy} holds even when~\eqref{eq:diverge-xy} is replaced by an arbitrary bilinear min-max problem $\min_{x\in \R^{d_x}} \max_{y \in \R^{d_y}} \; x^{\top} \bm B y$, for $\bm{B} \neq \bm 0$. 
\end{lemma}

\subsection{Chebyshev polynomials}\label{ssec:prelim:cheby}

Our analysis (only) for min-max problems with linear gradients in~\cref{sec:linear} makes use of classical facts about Chebyshev polynomials. We briefly recall relevant preliminaries here and refer the reader to standard textbooks such as~\citep{rivlin2020chebyshev,mason2002chebyshev} for further background. Throughout we use only Chebyshev polynomials \emph{of the first kind} and therefore drop that modifier for brevity.

\par Let $\cT_n$ denote the degree-$n$ Chebyshev polynomial over the interval $[-1,1]$. This can be defined in several equivalent ways, for example $2^{n-1}$ times the monic degree-$n$ polynomial minimizing $\|p\|_{L^{\infty}[-1,1]}$, or $2^{n-1}$ times the monic polynomial $\prod_{t=0}^{n-1} ( x - x_t)$ with roots $x_t = \cos(\tfrac{2t+1}{2n}\pi)$.
For a compact interval $[a,b]$, the degree-$n$ Chebyshev polynomial is defined via shifting as 
\[
\cT_n^{[a,b]}(\lambda) := \cT_n(\cL^{[a,b]}(\lambda)) \qquad \text{ where } \qquad \cL^{[a,b]}(\lambda) := \frac{2}{b-a} \left( \lambda - \frac{b+a}{2}\right)
\]
denotes the monotone linear map sending $[a,b]$ to $[-1,1]$. 

\par The relevance of Chebyshev polynomials for convex optimization is that the optimal GD stepsize schedules of length $T$ for optimizing $m$-strongly convex, $M$-smooth quadratic functions are (any permutation of) the $T$ inverse roots of $\cT_{T}^{[m,M]}$~\citep{young53}. Our results in~\cref{sec:linear} show how these ideas can also improve GDA for min-max optimization with bilinear or quadratic objectives.

%% file: sections/linear.tex
\section{Min-max problems with linear gradients}\label{sec:linear}

In this section we consider min-max problems $\min_x \max_y f(x,y)$ for which $\nabla f$ is linear. 
By using time-varying, asymmetric, and alternately negative stepsizes, we show that the iterates of GDA can be connected to polynomials of $\nabla^2 f$ (which in this setting is constant), and the optimal stepsizes of GDA are connected to the roots of certain extremal polynomials. This allows us to leverage classical techniques from convex quadratic optimization, thereby enabling a simple, explicit analysis of the proposed slingshot stepsize schedule. See~\cref{ssec:plausibility:bilinear} for a high-level overview.

\subsection{Bilinear objectives}\label{ssec:bilinear}

We first consider bilinear problems---a classical counterexample for GDA~\citep{samuelson1949market, korpelevich1976extragradient}.
Since we assume existence of a saddle point $(\tilde{x}, \tilde{y})$, these problems are of the form
\begin{equation}
	\min_{x\in \R^{d_x}} \max_{y \in \R^{d_y}} \; (x-\tilde{x})\T \bm B (y-\tilde{y}) \,.
	\label{eq:bilinear}
\end{equation}
By setting $\nabla f(z^*) = 0$, the saddle points $z^* = (x^*,y^*)$ are characterized by $x^* \in \tilde{x} + \Ker(\bm B^\top)$ and $y^* \in \tilde{y} + \Ker(\bm B)$. We make the standard assumption that $\bm B$ is spectrally bounded, namely all non-zero squared singular values of $\bm B$ lie in some interval $[m,M]$. The spectral upper bound can be interpreted as $M$-smoothness of the Hamiltonian $\Phi(z) = \tfrac{1}{2}\|\nabla f(z)\|^2$, or equivalently $\sqrt{M}$-smoothness of $f$. The spectral lower bound is equivalent to $\Phi$ being $m$-strongly convex; this can be relaxed, see~\cref{ssec:quad} for how the slingshot stepsize schedule still obtains the optimal convergence rate in that setting (i.e., where $f$ is any smooth bilinear function).

For such bilinear problems, we propose the following version of the slingshot stepsize schedule. Notice that the stepsizes are asymmetric, time-varying, and alternately positive or negative.

\begin{defin}[Slingshot stepsize schedules for bilinear min-max optimization]\label{def:steps-bilinear}
	For any even number of iterations $2T$ and any spectral bounds $0 < m \leq M < \infty$ on the bilinear coupling, the slingshot stepsize schedule is 
	\begin{align*}
		\alpha_{2t} = -\beta_{2t} = -\alpha_{2t+1} = \beta_{2t+1} = h_t\,, \qquad  t \in \{0,1,...,T-1\}\,,
	\end{align*}
	where $\{h_t\}_{t=0}^{T-1}$ are any permutation of $\{r_t^{-1/2}\}_{t=0}^{T-1}$, where 
	\begin{align*}
		r_t := \frac{M+m}{2} + \frac{M-m}{2} \cos\left( \frac{2t+1}{2T}\pi\right)\,, \qquad t \in \{0,1,...,T-1\}\,,
	\end{align*}
	are the $T$ roots of the Chebyshev polynomial $\cT_T^{[m,M]}$.
\end{defin}

The permutation invariance arises because the updates of GDA are commutative in this setting, at least in exact arithmetic, see~\cref{ssec:linear-alternative}. The relevance of $r_t$ (and hence the magnitudes of the stepsizes $h_t$) is that they are the $T$ roots of the solution to the following extremal polynomial problem, which as our analysis below shows, governs the fastest possible convergence rate of GDA. Below let $\cP_T$ denote the space of degree-$T$ polynomials $p$ satisfying $p(0) = 1$.

\begin{lemma}[Extremal polynomial lemma for bilinear min-max problems]\label{lem:extremal-bilinear}
	Let $T \in \N$, $0 < m < M < \infty$, and $\kappa := M/m$. Then 
	\begin{align*}
		\min_{p \in \cP_T} \max_{\lambda \in [m,M]} |p(\lambda)|
	\end{align*}
	has the unique optimal solution 
	\begin{align*}
	p(\lambda) = \frac{\cT_T^{[m,M]}(\lambda)}{\cT_T^{[m,M]}(0)} = \prod_{t=0}^{T-1} (1 - \lambda / r_t)
	\end{align*}
	and corresponding value
	\begin{align*}
		R_T
		:= \frac{1}{\abs{\cT_T^{[m,M]}(0)}} 
		=  \frac{2(\sqrt{\kappa} + 1)^T (\sqrt{\kappa} - 1)^T}{(\sqrt{\kappa}+1)^{2T} + (\sqrt{\kappa}-1)^{2T}}
		\approx \left(\frac{\sqrt{\kappa} - 1}{\sqrt{\kappa} + 1} \right)^T\,.
	\end{align*}
\end{lemma}

\cref{lem:extremal-bilinear} follows from the extremal definition of Chebyshev polynomials in~\cref{ssec:prelim:cheby}, see e.g.,~\citep[\S2]{d2021acceleration} or~\citep[\S2]{altschuler2018greed} for details. \cref{lem:extremal-bilinear} classically arises in convex optimization since the roots $r_t$ are the optimal inverse stepsizes for GD for $m$-strongly convex, $M$-smooth quadratic minimization, and $R_T$ is the optimal convergence rate for $T$ steps of any first-order optimization algorithm~\citep{young53}. 

We show that the slingshot stepsize schedule in~\cref{def:steps-bilinear} makes GDA converge. As described in~\cref{ssec:plausibility:bilinear}, using slingshot stepsizes with constant magnitudes $h_t$ is sufficient for exponentially fast convergence, and using the Chebyshev stepsize magnitudes in~\cref{def:steps-bilinear} accelerates the convergence to the optimal rate. For simplicity we state the result using distance as the performance criterion; this extends identically to convergence in gradient norm by either re-doing the same analysis or, more simply, by using the fact that distance and gradient norm are equivalent (up to a multiplicative factor depending only on $m$ and $M$) and thus switching these criteria does not affect the optimal asymptotic rate $\lim_{T \to \infty} R_T^{1/T}$.

\begin{theorem}[GDA converges for bilinear min-max optimization]\label{thm:bilinear-ub}
		Consider any integer $T$, any dimensions $d_x, d_y$, any initialization $z_0 = (x_0,y_0) \in \R^{d_x} \times \R^{d_y}$, and any bilinear min-max problem~\eqref{eq:bilinear} with spectral bounds $0 < m \leq M < \infty$. Using the slingshot stepsize schedule in~\cref{def:steps-bilinear}, $\GDA$ converges to a saddle point $z^*$ at rate
	\begin{equation}
		\norm{z_{2T} - z^*} \leq R_T \norm{z_0 - z^*}\,.
\nonumber
	\end{equation}
	In particular, $\|z_{2T} - z^*\| \leq \eps$ after $ \mathcal{O}(\sqrt{\kappa} \log \frac{\|z_0 - z^*\|}{\eps})$ steps. 
\end{theorem}
\begin{proof}
    Without loss of generality, after a possible translation we may assume that $\tilde{x} = 0$ and $\tilde{y} = 0$ to simplify notation.
    We show that the iterates converge to the saddle point $z^* = (x^*,y^*)$ that is closest to initialization; that is, let $x^*$ and $y^*$ be the projections of $x_0$ and $y_0$ onto $\Ker(\bm{B}^{\top})$ and $\Ker(\bm{B})$, respectively.
    For this problem, the $\GDA$ iterates form a linear dynamical system
	\begin{align*}
		\begin{bmatrix}x_{n+1} \\ y_{n+1} \end{bmatrix} = \bm{U_n} \begin{bmatrix}x_{n}\\ y_{n} \end{bmatrix} \qquad \text{ where } \qquad \bm{U_n} = \begin{bmatrix} \bm I & -\alpha_n \bm B \\ \beta_n \bm {B}\T & \bm I\end{bmatrix}\,.
	\end{align*}
	Because of the pairing of consecutive stepsizes in the slingshot stepsize schedule,
	\begin{align}
		\bm{U_{2t+1}} \bm{U_{2t}} 
		=
		\begin{bmatrix}
			\bm I & h_t \bm B \\
			h_t \bm{B}\T & \bm I
		\end{bmatrix}
		\begin{bmatrix}
			\bm I & -h_t\bm B \\
			-h_t \bm{B}\T & \bm I
		\end{bmatrix}
		=
		\begin{bmatrix}
			\bm {I} - h_t^2 \bm {BB}\T & \bm{0} \\
			\bm{0} & \bm {I} - h_t^2 \bm {B}\T\bm{ B} 
		\end{bmatrix} 
		\,.
		\label{eq:bilinear-pf-pairing}
	\end{align}
	Thus in particular
	\begin{align}
		\begin{bmatrix}x_{2T}\\ y_{2T}\end{bmatrix} 
		= 
		\prod_{t=0}^{T-1} \bm{U_{2t+1}} \bm{U_{2t}}
		\begin{bmatrix} x_{0} \\ y_{0}\end{bmatrix} 
		=  
		\begin{bmatrix}
			p(\bm {BB}\T) & \bm 0 \\ \bm 0 & p(\bm {B}\T\bm{B} )
		\end{bmatrix}
		\begin{bmatrix}x_{0}\\ y_{0}\end{bmatrix} 
		\,,
		\nonumber
	\end{align}
	where $p(z) := \prod_{t=0}^{T-1} (1 - h_t^2 z) = \prod_{t=0}^{T-1} (1 - z/r_t)$ is the extremal polynomial in~\cref{lem:extremal-bilinear}.
    Thus
    \begin{align*}
		\|y_{2T} - y^*\|
		=
		\|p(\bm{B}\T\bm{B}) (y_0 - y^*)\|
		= \| \bm V p(\bm{\Sigma}^2) \bm{V}\T (y_0 - y^*)\|
		\leq \|p(\bm{\Sigma}^2) \| \,  \|(y_0 - y^*)\|
		\leq R_T \, \|y_0 - y^*\|\,.
	\end{align*}
	Above, the first step is because $p(\bm{B}\T\bm{B}) y^* = \prod_{t=0}^{T-1} (\bm I - h_t^2 \bm{B}\T\bm{B}) y^* = y^*$ since $y^* \in \Ker(B)$. The second step is by an SVD decomposition of $\bm B = \bm U \bm \Sigma \bm{V}\T$, where $\bm \Sigma$ is invertible and $\bm U, \bm V$ are orthogonal. The third step is by sub-multiplicativity of the norm and orthogonality of $\bm V$. The fourth step is by~\cref{lem:extremal-bilinear} and the assumption that the non-zero singular values of $\bm B$ lie in $[\sqrt{m},\sqrt{M}]$, which implies that the non-zero eigenvalues of $\bm \Sigma^2$ lie in $[m,M]$. 
	
	\par By an identical argument (or symmetry), we also have
	\begin{align*}
		\|x_{2T} - x^*\| \leq R_T \|x_0 - x^*\|\,.
	\end{align*}
	Summing the squares of the above two displays yields $\|z_{2T} - z^*\|^2 \leq R_T^2 \|z_0 - z^*\|^2$. 
\end{proof}

The convergence rate $R_T$ in~\cref{thm:bilinear-ub} is exactly optimal not just among all possible GDA stepsize schedules, but also among all first-order algorithms. Matching lower bounds are known for the class of symmetric Krylov-subspace algorithms~\citep{azizian2020accelerating, Ibrahim_Azizian_Gidel_Mitliagkas_2020}, which are algorithms for which $(x_t,y_t) \in (x_0,y_0) + \mathrm{Span}\{ (-\nabla_x f(x_s,y_s), \nabla_y f(x_s,y_s)) : s < t\}$, and as stated below, this extends to asymmetric Krylov-subspace algorithms, which are algorithms for which $x_t  \in x_0 + \mathrm{Span}\{ \nabla_x f(x_s,y_s)  : s < t\}$ and $y_t  \in y_0 + \mathrm{Span}\{ \nabla_y f(x_s,y_s)  : s < t\}$. This generalization allows for asymmetric updates to $x$ and $y$, and thus captures the slingshot stepsizes (as well as other algorithms). The proof is a straightforward extension of the argument for symmetric algorithms, hence we defer it to~\cref{app:bilinear:lb}. 

\begin{theorem}[Optimality of~\cref{thm:bilinear-ub}]\label{thm:bilinear-lb}
	For any integer $T$, any initialization $z_0 = (x_0,y_0) \in \R^{d_x} \times \R^{d_y}$, and any asymmetric Krylov-subspace algorithm, there exists a bilinear problem~\eqref{eq:bilinear} for which the initialization is not a saddle point and 
	\begin{equation}
		\norm{z_{2T} - z^*}
		\geq
		R_T \norm{z_0 - z^*}
		\nonumber
	\end{equation}
	holds for all saddle points $z^* = (x^*, y^*)$.
\end{theorem}

An immediate corollary of our results is the resolution of an open question about more sophisticated variants of GDA. State-of-the-art complexity results show that extragradient, optimistic, and negative-momentum variants of GDA converge on bilinear problems but only at the unaccelerated rate $\mathcal{O}(\kappa \log 1/\epsilon)$~\citep{Zhang_Wang_2021,gidel2019negative,Daskalakis_Ilyas_Syrgkanis_Zeng_2018,Liang_Stokes_2019}, and it was open whether they could converge at the accelerated rate $\mathcal{O}(\sqrt{\kappa}\log 1/\epsilon)$.~\cref{thm:bilinear-ub} shows that GDA can achieve this accelerated rate, hence all these algorithms can too---for the trivial reason that GDA is a special case of these more sophisticated algorithms (use the slingshot stepsizes and set the momentum/look-ahead parameters to zero).

The work~\citep{azizian2020accelerating} also achieves the accelerated rate via algorithms that are specially tailored for the unconstrained bilinear setting.
Their algorithm effectively implements Polyak's heavy ball method on the Hamiltonian $\Phi := \tfrac{1}{2} \|\nabla f\|^2$ (which is a convex quadratic when $f$ is bilinear). Whereas, as explained in~\cref{ssec:connections:linear}, our algorithm effectively implements GD on the Hamiltonian with the Chebyshev stepsize schedule \citep{young53}---an equivalent way of accelerating quadratic minimization.
Hence the algorithm of~\citep{azizian2020accelerating} is to the slingshot stepsizes as Polyak's heavy ball method is to Young's stepsizes.

\par It is worth remarking that the slingshot stepsizes access the Hamiltonian in an entirely different way from~\citep{azizian2020accelerating}, which enables our algorithm to extend beyond bilinear $f$. The algorithms in~\citep[\S5]{azizian2020accelerating} rely on computing the operator $F^{\text{real}}(z) :=\tfrac{1}{\eta} [ \bm J\nabla f(z-\eta \bm J\nabla f(z)) - \bm J \nabla f(z) ]$ where $\bm{J} := \diag(\bm{I}, -\bm{I})$. For bilinear $f$, the diagonal blocks of $\nabla^2 f$ are zero, hence $F^{\text{real}}(z) = - \bm J\nabla^2 f(z) \bm J \nabla f(z)$ coincides with $\nabla \Phi = \nabla^2 f(z) \nabla f(z)$. However, beyond bilinear $f$, this is false since $- \bm J\nabla^2 f(z) \bm J \nabla f(z) \neq \nabla^2 f(z) \nabla f(z)$. In fact, $F^{\text{real}}$ can even be the \emph{negative} gradient of the Hamiltonian, which is not a direction that encourages stationarity---for example, for the quadratic problem $\min_x \max_y x^2$, this algorithm ascends rather than descends in $x$, leading to exponential divergence. By taking steps in the direction of the positive/negative gradient $\plusminus \nabla f$ rather than the monotone operator $\bm{J} \nabla f$, the slingshot stepsizes enable approximating $\nabla \Phi$ in general settings beyond bilinear $f$. See~\cref{sec:connections} for further discussion.

\subsection{Quadratic objectives}\label{ssec:quad}

We now turn to min-max problems with quadratic objectives. Since we assume the existence of a saddle point $\tilde{z} = (\tilde{x},\tilde{y})$, without loss of generality these problems are of the form
\begin{equation}\label{eq:quadminmax}
	\min_{x\in \R^{d_x}} \max_{y \in \R^{d_y}} \; \frac{1}{2}\begin{bmatrix}x - \tilde{x} \\ y - \tilde{y}\end{bmatrix}^\top
	\underbrace{\begin{bmatrix}\bm A & \bm B\\ \bm B^\top &  - \bm C\end{bmatrix}}_{\bm H}
	\begin{bmatrix}x - \tilde{x} \\ y - \tilde{y}\end{bmatrix}
\end{equation}
We assume that the objective is convex-concave ($\bm A, \bm C \succeq 0 $) and $L$-smooth ($-L \bm I \preceq \bm H \preceq L \bm I$). A primary motivation for this setting is that it captures bilinear problems without a spectral lower bound assumption, i.e., it captures arbitrary smooth bilinear problems by considering the special case $\bm A = 0$ and $\bm C = 0$. Since the slingshot stepsize schedule obtains optimal convergence rates for the more general setting of quadratic objectives (and not just smooth bilinear objectives\footnote{Note that the matching lower bound in~\cref{thm:quad-lb} is written for quadratic problems, yet the construction is a smooth bilinear problem, which establishes the optimality of the slingshot stepsizes for such problems.}), we present the results in this section at that higher level of generality.

\par For this setting we propose a similar slingshot stepsize schedule.

\begin{defin}[Slingshot stepsize schedule for quadratic min-max optimization]\label{def:quadsteps}
	For any even number of iterations $2T$ and any $L$-smooth, convex-concave, quadratic min-max problem, the slingshot stepsize schedule is 
	\begin{equation}
		\nonumber
		\alpha_{t} = -\beta_{t} = h_t,\qquad  t \in \{0,1,...,2T-1\},
	\end{equation}
	where $\{h_t^{-1}\}_{t=0}^{2T-1}$ are any permutation of 
	\begin{align}
\nonumber
		\rho_t := L \cos \left( \frac{2t+1}{4T+2} \pi \right)\,, \qquad t \in \{0,\dots,2T\} \setminus \{T\}\,,
	\end{align}
	which are the $2T$ non-zero roots of the Chebyshev polynomial $\cT_{2T+1}^{[-L,L]}$. Notice that like in~\cref{def:steps-bilinear}, these roots come in positive/negative pairs because 
    \[
        \rho_t = -\rho_{2T-t}\,.
    \]
\end{defin}

Just like the slingshot stepsize schedule for the bilinear setting in~\cref{def:steps-bilinear}, the stepsizes here are also asymmetric, time-varying, and paired positive/negative. 
The difference from the bilinear setting is in the stepsize magnitudes. This arises because the relevant extremal polynomial problem is slightly different in the quadratic setting here, namely the following.

\begin{lemma}[Extremal polynomial lemma for quadratic min-max problems]\label{lem:extremal-quad}
	For $T \in \N$ and $L > 0$,
	\begin{align*}
		\min_{p \in \cP_{2T}} \max_{\lambda \in [-L,L]} |\lambda p(\lambda)|
	\end{align*}
	has value $\frac{L}{2T+1}$, achieved by $p(\lambda) = \frac{(-1)^T}{2T+1} \frac{\cT_{2T+1}(\lambda/L)}{\lambda/L} = \prod_{t \in \{0,
		\dots,2T\} \setminus \{T\}} (1 - \lambda/\rho_t)$. 
\end{lemma}
\begin{proof}
    This polynomial $p$ achieves the stated value by the definition of Chebyshev polynomials, see~\cref{ssec:prelim:cheby}. It therefore suffices to prove a matching lower bound. Recall Bernstein's inequality $\sup_{x \in [-1,1]} |q(x)| \geq |q'(0)| / \mathrm{deg}(q)$ for any polynomial $q$~\citep[Theorem 5.17]{borwein2012polynomials}. Thus, for any $p \in \cP_{2T}$, the polynomial $q(x) := Lx p(Lx)$ has degree at most $2T+1$ and satisfies $q'(0) = L$, hence $\max_{\lambda \in [-L,L]} |\lambda p(\lambda)| = \max_{x \in [-1,1]} |q(x)| \geq L/(2T+1)$. 
\end{proof}

We show that these slingshot stepsizes enable GDA to converge for quadratic settings. As in the bilinear setting above, constant stepsize magnitudes $h_t$ are sufficient for convergence, and using the Chebyshev stepsize magnitudes in~\cref{def:quadsteps} accelerates the convergence to the optimal rate.

\begin{theorem}[GDA converges for quadratic min-max optimization]\label{thm:quad-ub}
	Consider any integer $T$, any dimensions $d_x, d_y$, any initialization $z_0 = (x_0,y_0)\in \R^{d_x}\times \R^{d_y}$, and any quadratic min-max problem~\eqref{eq:quadminmax}  that is convex-concave and $L$-smooth. Using the slingshot stepsize schedule in \cref{def:quadsteps}, GDA converges at rate
	\begin{equation}
		\|\nabla f(z_{2T})\| \leq \frac{L}{2T+1} \|z_0 - z^*\|\,,
	\end{equation}
    where $z^*$ is any saddle point. In particular, $\|\nabla f(z_{2T})\| \leq \eps$ after $\mathcal{O}(\frac{L \|z_0 - z^*\|}{\eps})$ steps. 
\end{theorem}
\begin{proof} 
    Without loss of generality, after a possible translation we may assume that $\tilde{x} = 0$ and $\tilde{y} = 0$ to simplify notation. 
    By the paired positive/negative property of the slingshot stepsize schedule, the iterates of GDA satisfy $z_{t+1} = (\bm{I} - h_t \bm{H})z_t$. Hence
	\begin{align*}
		\nabla f(z_{2T})
		= 
		\bm H 
		z_{2T}
		= 
		\bm H \prod_{t=0}^{2T-1} \left(\bm I - h_t \bm H \right)  (z_0 - z^*)
		= \bm H p(\bm H) (z_0 - z^*)\,,
	\end{align*}
	where $p(\lambda) = \prod_{t=0}^{2T-1}(1 - h_t \lambda) = \prod_{t=0}^{2T-1} (1 - \lambda/\rho_t)
	$ is the polynomial in~\cref{lem:extremal-quad}. By sub-multiplicativity, diagonalizing, and the extremal property of $p$ in~\cref{lem:extremal-quad}, we conclude that
	\begin{align*}
		\|\nabla f(z_{2T})\|
		\leq
		\norm{\bm H p(\bm H) }  \cdot\norm{z_0 - z^*}
		\leq
		\max_{\lambda \in [-L,L]} |\lambda p(\lambda)| \cdot \norm{z_0 - z^*}
		=
		\frac{L}{2T+1} \norm{z_0 - z^*}\,.
	\end{align*}
\end{proof}

Note that the convergence rate in~\cref{thm:quad-ub} holds for any saddle point $z^*$. The best bound is obtained by taking $z^*$ to be the closest saddle point to the initialization $z_0$, namely the projection of $z_0$ onto the affine subspace $\tilde{z} + \ker(H)$.

The convergence rate in~\cref{thm:quad-ub} is exactly optimal not just among all possible GDA stepsize schedules, but also among all first-order algorithms. Several other recent works also asymptotically achieve the accelerated convergence rate $\mathcal{O}(1/\epsilon)$ for $\|\nabla f(z_{2T})\|$ and do so in more general settings. However no prior work has obtained the optimal constants (as we do here), and moreover all prior methods require modifying GDA beyond just stepsizes, e.g., with anchoring  \citep{yoon2021accelerated,Yoon_Kim_Suh_Ryu_2024}.

\begin{theorem}[Optimality of \cref{thm:quad-ub} (Theorem 3 of~\citep{yoon2021accelerated})]\label{thm:quad-lb}
	For any integer $T$, any initialization $z_0 = (x_0, y_0)\in \R^{d_x}\times \R^{d_y}$, and any asymmetric Krylov-subspace algorithm, there exists an $L$-smooth quadratic min-max problem~\eqref{eq:quadminmax} for which the initialization is not a saddle point and
	\begin{align*}
		\|\nabla f(z_{2T})\| \geq \frac{L}{2T+1} \|z_0 - z^*\|
	\end{align*} 
	holds for all saddle points $z^* = (x^*,y^*)$. 
\end{theorem}

We conclude this section with two remarks.

\begin{remark}[Alternative performance criteria]\label{rem:quad-metrics}
	\cref{thm:quad-ub} bounds the gradient norm $\|\nabla f(z_{2T})\|$ in terms of the distance $\|z_0 - z^*\|$. In the bilinear setting,
    results extend to any combination of these performance criteria since distance and gradient norm are equivalent, see the discussion preceding~\cref{thm:bilinear-ub}.
    However, such extensions are false for the quadratic setting. Indeed, it is impossible to prove non-trivial rates (i.e., $1 - \delta$ for any $\delta > 0$) for the performance metrics $\frac{\|z_{2T} - z^*\|}{\| z_0 - z^*\|}$ and $\frac{\|z_{2T} - z^*\|}{\|\nabla f(z_0)\|}$ due to pathologically flat functions such as $f(x,y) = \gamma x^2$ at far initializations $z_0 = (x_0,0)$, where $\gamma > 0$ is sufficiently small. Similar counterexamples preclude non-trivial rates for $\tfrac{\norm{\nabla f(z_{2T})}}{\norm{\nabla f(z_0)}}$; consider the same construction, but now with $cx$ added to $f(x)$, for $c \neq 0$. 
\end{remark}

\begin{remark}[Faster convergence under stronger assumptions]
	\cref{thm:quad-ub} focuses on quadratic min-max problems in the smooth, (non-strongly) convex-concave setting. A line of work~\citep{azizian2020accelerating, Ibrahim_Azizian_Gidel_Mitliagkas_2020, Azizian_Mitliagkas_Lacoste-Julien_Gidel_2020} considers additional assumptions of strong convexity, strong concavity, and/or refined smoothness bounds on the individual blocks of $\bm H$, typically considering modifications of GDA using extragradients, momentum, optimism, etc. Such algorithms can sometimes automatically exploit further structure in quadratic objectives to obtain faster convergence. A potential direction for future work is to optimize stepsizes in these various settings in order to leverage additional such assumptions on the blocks of $\bm H$. Exponentially faster convergence can be obtained, for example, by appropriately adapting the slingshot stepsize schedule under the additional strongly-convex-strongly-concave assumption $\bm A, \bm C \succeq \mu \bm I$. See also the future work discussion in~\cref{sec:discussion}.
\end{remark}

\subsection{Alternative implementations and complex stepsizes}\label{ssec:linear-alternative}

For both the bilinear and quadratic settings, there are many equivalent implementations of the slingshot stepsize schedule that yield the same last iterate and thus the same convergence rate. This is due to invariances of the last iterate of GDA with respect to certain transformations of the stepsize schedule. For example, this includes arbitrary permutations of the stepsize schedule (since the GDA updates commute in the bilinear and quadratic settings). Another example is that for the bilinear setting, the stepsizes
$\alpha_{2t},\alpha_{2t+1},\beta_{2t},\beta_{2t+1}$ can be arbitrary solutions of the equations
\begin{align}\label{eq:css-bilinear:invariance}
    \alpha_{2t} + \alpha_{2t+1} =
    \beta_{2t} + \beta_{2t+1} = 0
    \qquad \text{ and } \qquad \alpha_{2t} \beta_{2t+1} = \alpha_{2t+1} \beta_{2t} = h_t^2 > 0\,,
\end{align}
since this is what is needed for the $2$-step identity~\eqref{eq:bilinear-pf-pairing} to hold in our analysis. These relate respectively to properties (iii) and (ii) of the general definition of slingshot stepsize schedules in~\cref{ssec:intro:slingshot}. Note that this system of equations~\eqref{eq:css-bilinear:invariance} has infinite solutions, not just $\alpha_{2t} = \beta_{2t+1} = h_t$ and $\beta_{2t} =\alpha_{2t+1} = -h_t$ as in~\cref{def:steps-bilinear}.

\par Perhaps the most striking equivalent implementation uses \emph{purely imaginary stepsizes}:
\begin{align}
	\alpha_{2t} = \beta_{2t} = i h_t \quad \text{ and } \quad \alpha_{2t+1} = \beta_{2t+1} = - i h_t
    \,, \qquad \forall t=0,1,2,\dots, T-1\,.
\end{align}
Note that this is a solution to~\eqref{eq:css-bilinear:invariance}.
This use of complex stepsizes is counterintuitive since the problem has purely real data. We are not aware of complex stepsizes being used anywhere else in optimization or min-max optimization. See the discussion of related work in~\cref{ssec:intro:prior}.
In this paper we focus on the real-stepsize implementation in~\cref{def:steps-bilinear} as it extends gracefully beyond bilinear objectives. (For general $f$, it is not even necessarily true that $f$ admits an analytic continuation.)

\par Another notable implementation uses random stepsizes, which are i.i.d.\ among pairs, and within pairs use the structure of the slingshot stepsize schedule. Unlike the aforementioned implementations which are all equivalent to each other, this produces a different rate for any finite number of iterations; however, it achieves the same optimal rate asymptotically; see~\cref{app:rand} for details.

\par Of course, all of this discussion of ``equivalent formulations'' requires exact arithmetic; in~\cref{app:stability} we isolate a particular implementation that is more numerically stable with respect to inexact updates and arithmetic.

%% file: sections/nonlinear.tex
\section{Min-max problems with nonlinear gradients}\label{sec:nonlinear}

We now turn to min-max problems $\min_x \max_y f(x,y)$ for which $\nabla f(x,y)$ is nonlinear. In this setting, the connection between GDA iterates and polynomials breaks down, necessitating a fundamentally different type of analysis from~\cref{sec:linear}. Nevertheless, we show that the same main results still hold: GDA can be made to converge by using time-varying, asymmetric stepsizes that are sometimes negative. The particular stepsizes we use are the following.

\begin{defin}[Slingshot stepsize schedule for convex-concave problems]\label{def:steps-nonlinear}
    Independently for each $t=0,...,T-1$, choose the stepsizes in iterations $2t$ and $2t+1$ as follows. With probability $1/2$,
    \begin{align*}
        \alpha_{2t} = -\beta_{2t} = \beta_{2t+1} = h \text{ \emph{and} }\alpha_{2t+1} = 0\,,
    \end{align*}
    and otherwise
     \begin{align*}
       -\alpha_{2t} =
        \beta_{2t} = \alpha_{2t+1} = h \text{ \emph{and} } \beta_{2t+1} = 0\,.
    \end{align*}
\end{defin}

This version of the slingshot stepsize schedule is largely similar to the version for min-max problems with linear gradients in~\cref{sec:linear}. A difference is that here we use randomness to break the symmetry for whether the $x$ or $y$ variable uses the negative stepsize in each batch of two iterations: $\beta_{2t} < 0$ in the first event, and $\alpha_{2t} < 0$ in the second event. We note that there are many versions of~\cref{def:steps-nonlinear} that lead to similar analyses and results as in this section. For example, we do not have to use only one negative step per two iterations. However, a key aspect of any convergent stepsize schedule (that is not satisfied by the stepsizes in~\cref{sec:linear}) is the following:

\begin{remark}[Positive net movement is necessary for convergence in convex-concave problems]\label{rem:huber}
    For the stepsizes in~\cref{def:steps-nonlinear}, $\E[\alpha_{2t} + \alpha_{2t+1}] = \E[\beta_{2t} + \beta_{2t+1}] > 0$. 
    If this were zero (as is the case for the optimal stepsizes for linear $\nabla f$ in~\cref{sec:linear}
    due to paired positive/negative stepsizes of equal magnitude),
    then GDA would not converge for smooth convex-concave problems due to pathological counterexamples that are locally linear. For example, consider $\min_{x \in \R} \max_{y \in \R} f(x,y) $ where $f(x,y)$ is the Huber loss function that is $\|x\|^2/2$ for $|x| \leq 1$ and otherwise $|x| - 1/2$. For GDA initializations bounded away from $[-1,1]$ for the $x$ variable, such schedules cycle since the positive/negative steps exactly cancel and lead to the same iterate in every other iteration.
\end{remark}

\par The use of negative stepsizes precludes the possibility of proving a ``$1$-step progress lemma'' for this algorithm. Nevertheless, following the same principle from the linear setting in~\cref{sec:linear}, here we establish progress after \emph{two} steps.

\begin{lemma}[2-step progress lemma]\label{lem:2step}
Let $f(x,y)$ be an $L$-smooth, $\mu$-strongly-convex-strongly-concave\footnote{
	Recall that a function $g$ is said to be $\mu$-strongly convex if $g(u) \geq g(v) + \langle \nabla g(v), u-v \rangle + \tfrac{\mu}{2}\|u-v\|^2$ for all $u,v$, and a function $f(x,y)$ is said to be $\mu$-strongly-convex-strongly-concave if $f(\cdot,y)$ and $-f(x,\cdot)$ are $\mu$-strongly convex for every $y$ and $x$. If $f$ is twice-differentiable (not required by our results), then this amounts to $\nabla_{xx}^2 f, -\nabla_{yy}^2 f \succeq \mu \bm{I}$. 
}
function with saddle point $z^* = (x^*,y^*)$. Then by using the slingshot stepsize schedule in~\cref{def:steps-nonlinear} with $h\leq \frac{1}{3L}$, the iterates of GDA satisfy
\begin{equation}\label{eq:lem2step}
\E\left[\|z_{2t+2}-z^*\|^2\right]\leq \left(1-h \mu\right)\E\left[\|z_{2t} - z^* \|^2\right] - \frac{h^2 (1-3Lh)}{2}\E\left[\|\nabla f(z_{2t})\|^2 \right]\,,
\end{equation}
where the expectation is over the randomness in the stepsize schedule.
\end{lemma}

This lemma applies to (non-strongly) convex-concave objectives by taking $\mu=0$. Below, we show how this lemma directly implies polynomial convergence of GDA for convex-concave objectives in~\cref{ssec:cc} and exponential convergence for strongly-convex-strongly-concave objectives in~\cref{ssec:scsc}, and then we prove the lemma in~\cref{ssec:2step}.

\subsection{Convex-concave objectives}\label{ssec:cc}

Here we show that the slingshot stepsizes enable GDA to converge on unconstrained smooth convex-concave problems---a classical counterexample for GDA that generalizes the bilinear setting. 

\begin{theorem}[GDA converges for convex-concave objectives]\label{thm:gd-cc}
	Let $f(x,y)$ be an $L$-smooth, convex-concave function with saddle point $z^* = (x^*,y^*)$. Then for any $T \in \N$, $h < \frac{1}{3L}$, and initialization $z_0 = (x_0,y_0)$, by using the slingshot stepsize schedule in~\cref{def:steps-nonlinear}, the iterates of GDA satisfy
	\begin{equation}
	\frac{1}{T} \sum_{t=0}^{T-1} \E\left[\|\nabla f(z_{2t})\|^2 \right] \leq \frac{2 \|z_0 - z^*\|^2}{h^2\left( 1-3Lh\right)T}\,.
	\end{equation}
    In particular, using stepsize parameter $h = \frac{c}{L}$ for any constant $c \in (0,\frac{1}{3})$ ensures
    \begin{align*}
         \frac{1}{T} \sum_{t=0}^{T-1} \E\left[\|\nabla f(z_{2t})\|^2 \right]  \lesssim \frac{L^2\|z_0 -z^*\|^2}{T}\,.
    \end{align*}
\end{theorem}

\begin{proof}
For $t=0,...,T-1$, applying~\cref{lem:2step} with $\mu = 0$ gives
\begin{equation}
    \frac{h^2(1-3 Lh)}{2} \,\E\left[\|\nabla f(z_{2t})\|^2 \right] 
    \leq
    \E\left[\|z_{2t} - z^* \|^2\right] - \E\left[\|z_{2t+2}-z^*\|^2\right] 
    \,.
\end{equation}
Summing these inequalities over $t$, telescoping, and crudely bounding $\E\|z_{2T} - z^*\|^2 \geq 0$ gives
\begin{align*}
    \frac{h^2 (1-3L h )}{2} \sum_{t=0}^{T-1}\E\left[\|\nabla f(z_{2t})\|^2 \right]
    \leq
    \E\left[\|z_{0} - z^* \|^2\right]
    \,.
\end{align*}
Re-arranging and dividing by $T$ completes the proof.
\end{proof}

\cref{thm:gd-cc} shows that on average, the iterates are nearly stationary. In practice, this can be implemented by choosing a random stopping time $\tau$ uniformly at random from $\{0, \dots, T-1\}$ and then outputting the iterate $z_{2\tau}$. Indeed,~\cref{thm:gd-cc} then implies that
\begin{equation}
\E\left[\|\nabla f(z_{2\tau})\|^2 \right] 
\lesssim \frac{L^2\|z_0 -z^*\|^2}{T}
\end{equation}
where the expectation is now over the choice of stopping time $\tau$ as well as the randomness in the slingshot stepsize schedule. Alternatively, one can output the ``best iterate'' in the trajectory with minimal $\|\nabla f(z_{2t})\|^2$; this enjoys the same performance guarantee in the worst case and can provide better performance in practice as it is adaptive to the problem. 

\begin{remark}[High-probability bound]
    A simple way to convert this convergence result from expectation to high probability is as follows. By~\cref{thm:gd-cc}, in expectation, the best iterate among $T$ iterations satisfies $\|\nabla f(z)\|^2 \leq R_T := \frac{2 \|z_0 - z^*\|^2}{h^2\left( 1-3Lh\right)T}$. Hence, by Markov's inequality, the best iterate satisfies $\|\nabla f(z)\|^2 \leq 2R_T$ with probability at least $1/2$. By a standard boosting argument, taking the best iterate over $\mathcal{O}(\log 1/\delta)$ runs amplifies this success probability to $1 - \delta$. 
\end{remark}

This $\mathcal{O}(1/T)$ rate matches popular algorithms that modify GDA beyond just stepsizes, e.g., extragradient \citep{Solodov_Svaiter_1999,gorbunov2022extragradient} and optimistic GDA \citep{Gorbunov_Taylor_Gidel, cai2022tight}, albeit for the best iterate rather than the last iterate. An accelerated rate of $\mathcal{O}(1/T^2)$ was recently achieved by a clever combination of extragradient and anchoring \citep{yoon2021accelerated}, and it is an interesting question if our results can be similarly accelerated, see the discussion in~\cref{sec:discussion}.~\cref{fig:cc-convergence} provides an illustrative numerical comparison of these algorithms.

\begin{figure}[ht]
    \centering
    \begin{subfigure}[t]{0.47\linewidth}
        \centering
        \includegraphics[width=\linewidth]{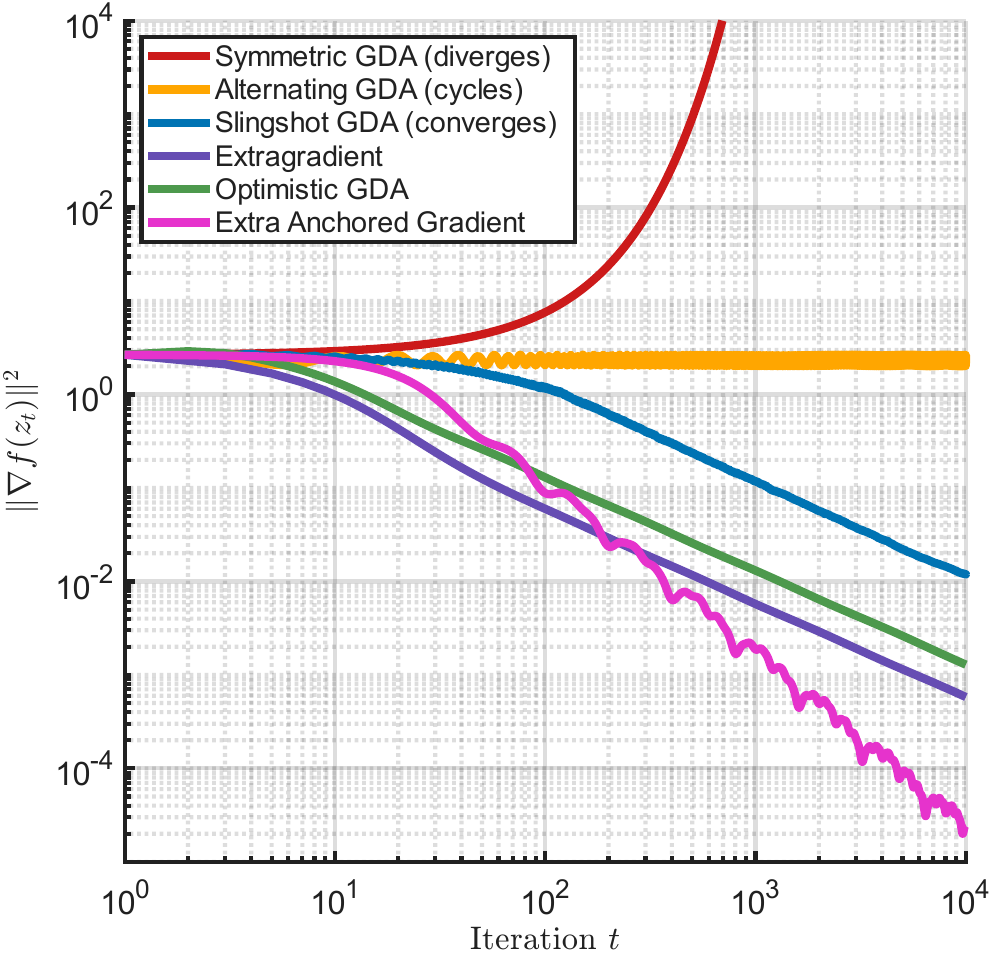}
        \caption{\footnotesize 
        Empirical comparison of the convex-concave version of the slingshot stepsize schedule (\cref{def:steps-nonlinear}).
        }
        \label{fig:cc-subfig-a}
    \end{subfigure}
    \hfill
    \begin{subfigure}[t]{0.47\linewidth}
        \centering
        \includegraphics[width=\linewidth]{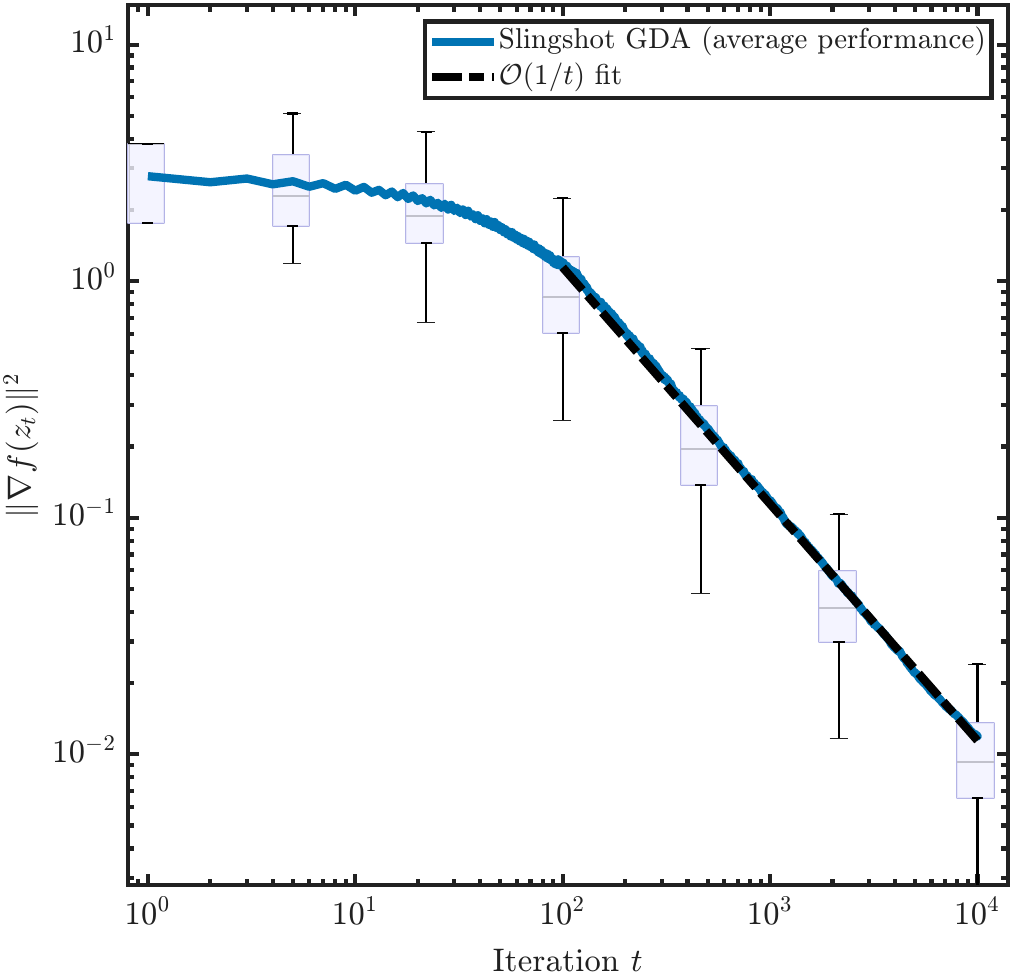}
        \caption{\footnotesize 
        Variability of slingshot stepsizes across $1000$ runs. The $\mathcal{O}(1/\epsilon)$ empirical rate matches the theory in~\cref{ssec:cc}.}
        \label{fig:cc-subfig-b}
    \end{subfigure}
    \caption{\footnotesize 
        Illustrative numerical comparison on a $1$-smooth bilinear problem $\min_{x \in \mathbb{R}^{100}} \max_{y \in \mathbb{R}^{100}} x^\top \bm{B} y$ where $\bm{B} = \text{diag}(2^0, 2^{-1}, \dots, 2^{-99})$. All algorithms are initialized at the all-ones vector. Regardless of the stepsize choice, symmetric GDA diverges and alternating GDA cycles. Here we run the slingshot stepsize schedule for general convex-concave problems (\cref{def:steps-nonlinear}); using the version tailored to smooth bilinear problems (\cref{def:quadsteps}) leads to much faster convergence than all plotted algorithms, with $\|\nabla f(z_t)\|^2 \approx 5 \times 10^{-7}$ at the last iterate $t = 10^4$. We set $h=2/(9L)$ in \cref{def:steps-nonlinear} which minimizes the rate in~\cref{thm:gd-cc}. Extragradient, optimistic GDA, and extra anchored gradient are run with their standard stepsize prescriptions. See~\cref{app:experimental-setup} for further details. 
        }
    \label{fig:cc-convergence}
\end{figure}

\subsection{Strongly-convex-strongly-concave objectives}\label{ssec:scsc}

In the case of strongly-convex-strongly-concave objectives, GDA converges with constant positive stepsizes (this is the only such setting in the paper), but at a slow rate---namely $\Theta(\kappa^2 \log 1/\eps)$ is the fastest possible rate for symmetric GDA~\citep{azizian2020accelerating}, improvable to $\mathcal{O}(\kappa^{1.5} \log 1/\eps)$ using alternating GDA~\citep{Lee_Cho_Yun_2024}.\footnote{GDA with non-negative alternating stepsizes can asymptotically match the accelerated rate $\Theta(\kappa \log 1/\eps)$, but only if one considers local convergence from arbitrarily close initialization and additionally assumes that the Hessian is continuous in this local neighborhood around optimality~\citep{Zhang_Wang_Lessard_Grosse_2022}. \cref{thm:gd-scsc} shows that the slingshot stepsizes achieve this optimal rate globally and without any such additional assumptions.} Here we show that the slingshot stepsize schedule achieves the accelerated rate $\Theta(\kappa \log 1/\eps)$, which matches extragradient and optimistic GDA~\citep{mokhtari2020unified}, and moreover is optimal among arbitrary first-order methods without further assumptions on the problem  setting~\citep{Azizian_Mitliagkas_Lacoste-Julien_Gidel_2020}.

\begin{theorem}[GDA converges for strongly-convex-strongly-concave objectives]\label{thm:gd-scsc}
	Let $f(x,y)$ be an $L$-smooth, $\mu$-strongly-convex-strongly-concave function with saddle point $z^* = (x^*,y^*)$. Denote $\kappa := L/\mu$. Then for any $T \in \N$ and any initialization $z_0 = (x_0,y_0)$, by using the stepsize schedule in~\cref{def:steps-nonlinear} with parameter $h \leq \frac{1}{3L}$, the iterates of GDA satisfy
	\begin{equation}
	\E\left[\|z_{2T} - z^*\|^2 \right] \leq \left( 1-h \mu\right)^T\|z_0 - z^*\|^2\,.
	\end{equation}
	In particular, setting $h = \tfrac{1}{3L}$, then $\E\left[\|z_{2T} - z^*\|^2 \right]\leq \eps$ after $\cO(\kappa\log \frac{\|z_0 - z^*\|^2}{\eps})$ steps. 
\end{theorem}
\begin{proof}
Applying~\cref{lem:2step} and dropping the gradient term (as it is non-negative by the assumption on $h$) yields the contraction
$\E[\|z_{2t+2}-z^*\|^2]\leq (1-h \mu)\E[\|z_{2t} - z^* \|^2]$. Recurse $T$ times.
\end{proof}

\begin{remark}[High-probability bound]
    A simple way to convert this convergence result from expectation to high probability is as follows. By Markov's inequality, the last iterate satisfies $\|z_{2T} - z^*\|^2 \leq \eps$ with probability at least $1 - \delta$ if\ $\E \|z_{2T} - z^*\|^2 \leq \eps\delta$. By~\cref{thm:gd-scsc}, this is ensured by running $\mathcal{O}(\kappa \log 1/\delta)$ additional iterations.
\end{remark}

\subsection{Proof of $2$-step lemma (simplified)}\label{ssec:2step}

We now prove the 2-step progress guarantee underlying~\cref{lem:2step}. To keep the main argument transparent, we prove here a slightly weaker version of the lemma, in which the allowable stepsize range $h \lesssim 1/L$ and the coefficient on the progress term $\|\nabla f(z_{2t})\|^2$ are worse by universal constants. This weaker estimate is already sufficient to obtain the convergence guarantees in~\cref{thm:gd-cc,thm:gd-scsc} up to universal constants. The sharper constants stated in~\cref{lem:2step} are proved in~\cref{app:2-step} using the framework of performance estimation problems (PEP). The PEP analysis is included not only to improve constants, but also to illustrate how PEP can help identify the right combination of smoothness, convexity-concavity, and auxiliary-point inequalities; in particular, the ability of PEP to systematically search over convergence proofs was helpful for our design of the slingshot stepsize schedule and directly motivated the simpler, more interpretable proof below.

To simplify notation, without loss of generality let $t = 0$ (so that we measure progress from $z_0$ to $z_2$) and let $L=1$ (by rescaling the objective and stepsize). Recall that the randomness of the proposed stepsize schedule in~\cref{def:steps-nonlinear} leads to two possible trajectories, each with probability $1/2$. Denote these two trajectories as
\begin{equation}\label{eq:2step-traj1}
\begin{aligned}
x_{1}^{(-)} &= x_{0} - h \nabla_x f(x_0,y_0),
\qquad &
x_{2}^{(-)} &= x_{1}^{(-)}\,,
\\
y_{1}^{(-)} &= y_{0} - h \nabla_y f(x_0, y_0),
\qquad  &y_{2}^{(-)}&= y_{1}^{(-)} +h  \nabla_y f\left(x_1^{(-)}, y_1^{(-)}\right)\,,
\end{aligned}
\end{equation}
and
\begin{equation}\label{eq:2step-traj2}
\begin{aligned}
x_{1}^{(+)} &= x_{0} + h \nabla_x f(x_0,y_0),
\qquad
&
x_{2}^{(+)} &= x_{1}^{(+)}-h \nabla_x f\left(x_1^{(+)}, y_1^{(+)}\right)\,,
\\
y_{1}^{(+)} &= y_{0} + h  \nabla_y f(x_0, y_0)
,
\qquad  &y_{2}^{(+)}&= y_{1}^{(+)} \,.
\end{aligned}
\end{equation}
We prove the inequality
\begin{equation}\label{eq:2step-pf}
 \frac{1}{2} \left( \left\|z_{2}^{(-)}-z^*\right\|^2 + \left\|z_{2}^{(+)}-z^*\right\|^2 \right) \leq \left(1- h \mu\right) \|z_0 - z^*\|^2 - \frac{h^2}{2} \left(1-6h + 2h\mu-2h^2\right) \left\|\nabla f(z_0)\right\|^2 \,,
\end{equation}
where for shorthand $z_0 := (x_0,y_0)$, $z_2^{(-)} := \big(x_2^{(-)}, y_2^{(-)}\big)$ and $z_2^{(+)} := \big(x_2^{(+)}, y_2^{(+)}\big)$. This inequality~\eqref{eq:2step-pf} implies the 2-step progress in~\cref{lem:2step} except with slightly weaker constants; indeed, $1-6h + 2h\mu-2h^2 \ge 1-7h$ for $h \leq 1/7$, whereas the analogous term (after normalizing $L=1$) in~\cref{lem:2step} is $1-3h$ for $h \leq 1/3$. See~\cref{app:2-step} for a proof of~\cref{lem:2step} with the sharper stated constants.

\paragraph*{Proof of~\eqref{eq:2step-pf}.} Begin by expanding the squared norms:
\begin{equation} \nonumber
\begin{aligned}
\left\|z_2^{(+)}-z^*\right\|^2
&= \left\|z_1^{(+)}-z^*\right\|^2
  +h^2\left\|\nabla_x f(z_1^{(+)})\right\|^2
  +2h\left\langle \nabla_x f(z_1^{(+)}), x^*-x_1^{(+)} \right\rangle\,,\\
\left\|z_2^{(-)}-z^*\right\|^2
&= \left\|z_1^{(-)}-z^*\right\|^2
  +h^2\left\|\nabla_y f(z_1^{(-)})\right\|^2
  +2h\left\langle \nabla_y f(z_1^{(-)}), y_1^{(-)}-y^* \right\rangle\,.
\end{aligned}
\end{equation}
Next, we bound these inner-product terms using (strong) convexity-concavity of $f$. Consider first the $(+)$ trajectory. By $\mu$-strong convexity of $f(\cdot,y_1^{(+)})$ and then $\mu$-strong concavity of $f(x^*,\cdot)$,
\begin{align*}
    \left\langle \nabla_x f(z_1^{(+)}), x^*-x_1^{(+)} \right\rangle
    &\leq f(x^*,y_1^{(+)}) - f(z_1^{(+)}) - \frac{\mu}{2}\|x_1^{(+)} - x^*\|^2
    \\ &\leq f(z^*) - f(z_1^{(+)}) - \frac{\mu}{2}\|z_1^{(+)} - z^{*}\|^2\,.
\end{align*}
An analogous argument for the $(-)$ trajectory gives
\begin{align*}
    \left\langle \nabla_y f(z_1^{(-)}), y_1^{(-)}-y^* \right\rangle
    &\leq
    f(z_1^{(-)}) - f(x_1^{(-)},y^*)
    - \frac{\mu}{2} \|y_1^{(-)} - y^*\|^2
    \\ &\leq 
    f(z_1^{(-)}) - f(z^*)
    - \frac{\mu}{2} \|z_1^{(-)} - z^*\|^2
    \,.
\end{align*}
Combining the above three displays yields
\begin{align}
& \frac{1}{2}\left(
  \left\|z_2^{(+)}-z^*\right\|^2
  +\left\|z_2^{(-)}-z^*\right\|^2
\right) \label{eq:average-bound}
\\ & \leq (1 - h\mu) \cdot \underbrace{\frac{1}{2}\left(
  \left\|z_1^{(+)}-z^*\right\|^2
  +\left\|z_1^{(-)}-z^*\right\|^2
\right)}_{\circled{1}}
+\frac{h^2}{2}\underbrace{\left(
  \left\|\nabla_x f(z_1^{(+)})\right\|^2
  +\left\|\nabla_y f(z_1^{(-)})\right\|^2
\right)}_{\circled{2}}
+ h\underbrace{\bigl(f(z_1^{(-)}) -f(z_1^{(+)})\bigr)}_{\circled{3}}\,. \nonumber
\end{align}
We bound $\circled{1},\circled{2},\circled{3}$ to obtain the right hand side of~\eqref{eq:2step-pf}. For $\circled{1}$, we have the identity
\begin{align*}
\circled{1} = \frac{1}{2}\left(
  \left\|z_1^{(+)}-z^*\right\|^2
  +\left\|z_1^{(-)}-z^*\right\|^2
\right) =  \|z_0-z^*\|^2 + h^2\|\nabla f(z_0)\|^2.
\end{align*}
To bound $\circled{2}$, use the algorithm definition, use $1$-Lipschitzness of $\nabla f$, and simplify:
\begin{align*}
\circled{2}
  &= \left\|\nabla_x f(z_0+h\nabla f(z_0))\right\|^2
+\left\|\nabla_y f(z_0-h\nabla f(z_0))\right\|^2\\
&\leq
\left(
  \left\|\nabla_x f(z_0)\right\|
  + h %
  \left\|\nabla f(z_0)\right\|
\right)^2
+
\left(
  \left\|\nabla_y f(z_0)\right\|
  + h %
  \left\|\nabla f(z_0)\right\|
\right)^2\\
&\leq
\left(1+4h+2h^2\right)
\left\|\nabla f(z_0)\right\|^2.
\end{align*}
Finally, to bound $\circled{3}$, use the algorithm definition, use $1$-smoothness of $f$, and simplify:
\begin{align*}
\circled{3}
&= f(z_0 - h\nabla f(z_0)) - f(z_0+h\nabla f(z_0))\\
&\leq \left(f(z_0) - h\|\nabla f(z_0)\|^2 + \frac{h^2}{2}\|\nabla f(z_0)\|^2\right)-\left(f(z_0) + h\|\nabla f(z_0)\|^2 - \frac{h^2}{2}\|\nabla f(z_0)\|^2\right)\\
&= (-2h+h^2)\|\nabla f(z_0)\|^2.
\end{align*}
Plugging these bounds for $\circled{1},\circled{2},\circled{3}$ into~\eqref{eq:average-bound} yields the desired inequality~\eqref{eq:2step-pf}.

%% file: sections/connections.tex
\section{Connections to consensus optimization and finite differencing}\label{sec:connections}

Here we interpret the slingshot stepsize schedule through the second-order expansion of its dynamics. As discussed in~\cref{ssec:intro:cont}, this provides a complementary interpretation of the key two-step progress phenomenon and applies at a higher level of generality (the Taylor expansion does not require convexity-concavity), although we do not extract convergence guarantees in this section. This interpretation provides connections to finite-differencing techniques as well as to Hamiltonian gradient descent and the consensus optimization algorithm. In order to focus on the main conceptual ideas, the derivations in this section are kept informal and we assume for simplicity that $f \in \mathcal{C}^2$. (Note that none of our results in~\cref{sec:nonlinear} require this assumption; in particular, our results only require the standard assumption that $\nabla f$ is Lipschitz continuous.) For concreteness and simplicity, we focus on the implementations of the slingshot stepsize schedule in~\cref{sec:linear} and~\cref{sec:nonlinear}; see~\cref{app:second-order} for the general case. See also~\cref{ssec:plausiblity:approx} for a high-level overview of the derivations in this section.

\subsection{Linear version of the slingshot stepsize schedule}\label{ssec:connections:linear}

\paragraph{Second-order expansion.}
Recall from~\cref{sec:linear} that for problems with linear gradients $\nabla f$, we proposed versions of the slingshot stepsize schedule that use paired positive/negative steps of equal magnitude (see \cref{def:steps-bilinear} and \cref{def:quadsteps} for the settings of bilinear and quadratic objectives, respectively). By concatenating $z_t = (x_t,y_t)$, this amounts to a two-step update which alternates moving in a positive and negative gradient direction of $f$:
\begin{equation}\label{eq:lin_css_update}
\begin{aligned}
z_{2t+1}&= z_{2t}+h\nabla f(z_{2t}) \\ z_{2t+2}&= z_{2t+1}-h\nabla f(z_{2t+1}).
\end{aligned}
\end{equation}
Although this version of the slingshot stepsize schedule is designed specifically for the setting of linear $\nabla f$, it is insightful to consider the dynamics for general objectives with nonlinear $\nabla f$. Specifically, below we expand the two-step update from $z_{2t}$ to $z_{2t+2}$ for general nonlinear $\nabla f$. (For linear $\nabla f$, as in~\cref{sec:linear}, all the following approximations are equalities.) A Taylor approximation gives 
\begin{align}
    \nabla f(z_{2t+1}) = \nabla f(z_{2t} + h \nabla f(z_{2t})) 
    \approx \nabla f(z_{2t}) +h \nabla^2 f(z_{2t}) \nabla f(z_{2t})
    \label{eq:lin_css_taylor}
\end{align}
up to higher-order error terms.
Plugging this into~\eqref{eq:lin_css_update} yields 
\begin{equation}\label{eq:2nd_order_expansion_linear}
z_{2t+2}
=
z_{2t} + h \nabla f(z_{2t}) - h \nabla f(z_{2t+1})
\approx
z_{2t}-h^2 \nabla^2 f(z_{2t})\nabla f(z_{2t})
\end{equation}
up to higher-order error terms.
Note that this expansion~\eqref{eq:2nd_order_expansion_linear} is invariant with respect to the order of the positive/negative steps in~\eqref{eq:lin_css_update}. 

\paragraph{Geometric interpretation.} The second-order expansion~\eqref{eq:2nd_order_expansion_linear} quantifies the \emph{non-reversibility} of gradient descent: performing a step of gradient descent followed by a step of gradient ascent, both of the same stepsize $h$, does not return to the starting point $z$. Instead, there is a second-order drift $-h^2 \nabla^2 f(z) \nabla f(z)$. 
This drift's direction $\nabla^2 f(z) \nabla f(z)$ is the directional derivative of the gradient along itself (described more below) and can be interpreted as the \emph{acceleration} of gradient ascent or gradient descent; this is perhaps most easily seen in continuous-time: 
\begin{align*}
    \frac{d}{dt}z(t) = \plusminus \nabla f(z(t)) 
  \qquad \Longrightarrow  \qquad \frac{d^2}{dt^2} z(t) = \nabla^2 f(z(t)) \nabla f(z(t))\,.
\end{align*}
Geometrically, the invariance of the second-order expansion~\eqref{eq:2nd_order_expansion_linear} with respect to the order of the ascent and descent step amounts to the fact that the Lie bracket of the vector fields $\nabla f$ and $-\nabla f$ is zero, i.e., the commutator $[\nabla f, -\nabla f]$ vanishes.
 
\paragraph{Connection to HGD} The second order expansion~\eqref{eq:2nd_order_expansion_linear} of this two-step update is equivalent to one step of Hamiltonian gradient descent (HGD).
HGD is an algorithm used for solving multiplayer games \citep{Balduzzi_Racaniere_Martens_Foerster_Tuyls_Graepel_2018, Mescheder_Nowozin_Geiger_2017,Abernethy_Lai_Wibisono_2021} and amounts to taking gradient descent steps in the concatenated variable $z = (x,y)$ with respect to the Hamiltonian function $\Phi(z)=\frac{1}{2}\|\nabla f(z)\|^2$, the idea being that moving in a descent direction of the Hamiltonian encourages stationarity. Note that both variables minimize the same potential, effectively colluding rather than competing as in the original min-max problem. The update of HGD is given by
\begin{equation}\label{eq:hgd}
z_{t+1} = z_t - \gamma\nabla \Phi(z_t) = z_t - \gamma \nabla^2 f(z_t) \nabla f(z_t)\,,
\end{equation}
where the second equality is by the chain rule. Observe that this HGD update~\eqref{eq:hgd} exactly matches the second-order approximation~\eqref{eq:2nd_order_expansion_linear} of the slingshot stepsize schedule.

\par HGD is known to converge quickly in some settings, for example on functions whose Hamiltonians are smooth and satisfy a Polyak-Lojasiewicz condition~\citep{Abernethy_Lai_Wibisono_2021}. However, these conditions on the Hamiltonian do not translate to standard assumptions on $f$; indeed HGD fails to converge even on simple objectives that are smooth and convex-concave, such as the $1$-dimensional counterexample in~\cref{rem:huber}: consider the Huber function $f(x,y)$ that is $x^2/2$ for $|x| \leq 1$ and otherwise is $|x| - 1/2$. If HGD is initialized outside of $[-1,1]$, then $\nabla \Phi = 0$, hence HGD updates never move.

\paragraph{Connection to finite differencing.} Note that the GDA algorithm uses only first-order information, yet effectively yields a two-step update~\eqref{eq:2nd_order_expansion_linear} that incorporates second-order information. This use of higher-order information is made possible because the  slingshot stepsize schedule enables GDA to implicitly implement finite differencing. In particular, here the relevant second-order information is the directional derivative $D_{\nabla f(z) / \|\nabla f(z)\|} \nabla f(z)$ of the function $\nabla f(\cdot)$ in the direction $\nabla f(z) / \|\nabla f(z)\|$, evaluated at $z$. That is,
\begin{equation}\label{eq:finite_diff_approx}
D_{\nabla f(z) / \|\nabla f(z)\|} \nabla f(z)
=
\nabla^2 f(z)\frac{\nabla f(z)}{\|\nabla f(z)\|} \approx \frac{\nabla f(z + h\nabla f(z)) - \nabla f(z)}{h\|\nabla f(z)\| }\,.
\end{equation}
Indeed, this is precisely the key Taylor expansion in~\eqref{eq:lin_css_taylor}, modulo rearranging and rescaling. 

\par Conceptually, this use of finite differencing is similar to how zeroth-order optimization algorithms compute directional derivatives of a function $f$ using finite-difference approximations \citep{nesterov2017random, conn2009introduction}. A subtle but important distinction is that in zeroth-order methods, the finite-differencing approximation $D_{u / \|u\|}\nabla f(z) \approx \frac{\nabla f(z +  u) - \nabla f(z)}{\|u\|}$ uses two evaluations within a small distance $\|u\| \ll 1$, which leads to simple bounds for the approximation~\citep{nesterov2017random}. In contrast, the slingshot stepsize schedule approximates the directional derivative of $\nabla f(z)$ by using two evaluations within distance $h \|\nabla f(z)\|$, and unless it is already known that the algorithm is converging, this quantity is not actually known to be small. This necessitates different types of analyses than is standard in the finite-differencing literature.

\subsection{Nonlinear version of the slingshot stepsize schedule}\label{ssec:connections:nonlinear}

\paragraph{Second-order expansion.} We now provide the analogous expansion for two steps of GDA using the slingshot stepsize schedule for nonlinear $\nabla f$ in~\cref{def:steps-nonlinear}. Recall from~\cref{sec:nonlinear} that this schedule uses randomization to break the symmetry between whether $x$ or $y$ uses a negative stepsize. Following the notation in~\cref{ssec:2step}, let us distinguish these two trajectories using superscripts $(+)$ and $(-)$. The first step of GDA in these two trajectories is $z_{2t+1}^{(+)} = z_{2t}+h \nabla f(z_{2t})$ and $z_{2t+1}^{(-)} = z_{2t}-h \nabla f(z_{2t})$. Identical to~\eqref{eq:lin_css_taylor}, a Taylor approximation of $\nabla f$ in both cases yields
\begin{equation}
	\nabla f\left(z_{2t+1}^{(+)}\right) \approx \nabla f(z_{2t}) +h\nabla^2 f(z_{2t}) \nabla f(z_{2t}),\quad \nabla f\left(z_{2t+1}^{(-)}\right) \approx \nabla f(z_{2t}) -h\nabla^2 f(z_{2t}) \nabla f(z_{2t}).
\end{equation}
Plugging this gradient approximation into the second step's update gives the two-step expansion:
\begin{equation}\label{eq:2nd_order_expansion_seperated}
	\begin{aligned}
		&x_{2t+2}^{(+)} \approx x_{2t}- h^2 \left[\nabla^2 f(z_{2t}) \nabla f(z_{2t}) \right]_x,\quad &x_{2t+2}^{(-)} &= x_{2t}- h\nabla_x f(z_{2t})\,, \\
		&y_{2t+2}^{(+)} = y_{2t}+h \nabla_y f(z_{2t}),\quad &y_{2t+2}^{(-)} &\approx y_{2t}-h^2\left[\nabla^2 f(z_{2t}) \nabla f(z_{2t}) \right]_y .
	\end{aligned}
\end{equation}
Above, for a vector $v := (v_x,v_y) \in \R^{d_x + d_y}$, we write $[v]_x$ to denote $v_x$, and similarly for the $y$ component. Observe from this two-step expansion~\eqref{eq:2nd_order_expansion_seperated} that in the $(+)$ trajectory, the $x$ variable effectively performs one step of HGD, and the $y$ variable performs one step of gradient ascent; and vice versa in the $(-)$ trajectory. 
As a consequence, the \emph{expected} two-step update is
\begin{equation}\label{eq:2nd_order_expansion}
	\E[z_{2t+2}]
	= \frac{z_{2t+2}^{(+)} + z_{2t+2}^{(-)}}{2} 
	\approx z_{2t} - \frac{h}{2}\bm J \nabla f(z_{2t}) - \frac{h^2}{2}\nabla^2 f(z_{2t}) \nabla f(z_{2t})
\end{equation}
up to higher-order error terms,
where $\bm J := \begin{bmatrix} \bm I & \bm 0\\ \bm 0 &-\bm I \end{bmatrix}$. 

\par The key point is that this second-order expansion~\eqref{eq:2nd_order_expansion} can be viewed as a convex combination of two distinct updates:
\begin{itemize}
	\item HGD update: $z_{2t} - h^2 \nabla^2 f(z_{2t}) \nabla f(z_{2t})$.
	\item GDA update: $z_{2t} - h \bm{J} \nabla f(z_{2t})$. (Note that this GDA update uses conventional stepsizes: they are positive and symmetric for $x$ and $y$).
\end{itemize}
It is crucial that the update~\eqref{eq:2nd_order_expansion} incorporates both GDA and HGD, since by themselves, both algorithms fail to converge even on simple unconstrained min-max problems (see~\cref{ssec:warmup:diverge} and~\cref{ssec:connections:linear} for simple 1-dimensional counterexamples where GDA and HGD fail, respectively). 

\paragraph{Connection to consensus optimization.} Interestingly, the second-order expansion~\eqref{eq:2nd_order_expansion} of this two-step update is equivalent to the consensus optimization algorithm of~\citep{Mescheder_Nowozin_Geiger_2017}, which updates using a combination of a GDA and HGD step, namely
\begin{equation}\label{eq:consensus-opt}
z_{t+1} = z_t - h\bm J \nabla f(z_t) - \gamma\nabla \Phi(z_t).
\end{equation}
This algorithm has received significant attention in the machine learning community due to its success for training Generative Adversarial Networks (GANs), which is a notoriously challenging setting for min-max algorithms, and in particular is a setting where both GDA and HGD experience trouble converging by themselves~\citep{Mescheder_Nowozin_Geiger_2017}.

\par Although elegant convergence results have been proven for consensus optimization, this area is still nascent. Existing results either show non-explicit asymptotic convergence that requires arbitrarily close initialization and growth conditions around optimality~\citep{Mescheder_Nowozin_Geiger_2017}, or provide non-asymptotic convergence bounds but only apply to bilinear $f$~\citep{Liang_Stokes_2019} (and moreover with unaccelerated rates) or require non-standard smoothness and/or Polyak-Lojasiewicz-type assumptions on the Hamiltonian which do not apply for smooth convex-concave $f$ (or even smooth strongly-convex-strongly-concave $f$)~\citep{Azizian_Mitliagkas_Lacoste-Julien_Gidel_2020,Abernethy_Lai_Wibisono_2021}. For example, $f(x,y) = \log \cosh x$ is convex-concave, yet its Hamiltonian $\Phi(x,y) = \frac{1}{2}\|\nabla f(x,y)\|^2 = \frac{1}{2}\tanh^2(x)$ does not satisfy the Polyak-Lojasiewicz condition required by~\citep{Abernethy_Lai_Wibisono_2021}. Therefore, when we prove our convergence results for the setting of smooth convex-concave $f$ in~\cref{sec:nonlinear}, rather than building upon this connection to consensus optimization, we instead directly prove progress for two steps of the slingshot stepsizes. (In the other direction, we are hopeful that our convergence results and techniques may be helpful for analyzing consensus optimization.)

\paragraph{Connections to finite differencing.} Just like the linear version of the slingshot stepsize schedule (see~\cref{ssec:connections:linear}), the nonlinear version here implicitly implements finite-differencing in every pair of steps, enabling GDA to effectively move in directions involving second-order information about the objective, despite only ever using first-order information in its updates. 
See~\eqref{eq:2nd_order_expansion}. 

\par This way of computing second-order information using finite differencing is different from how consensus optimization and HGD are implemented. Even though both those algorithms require updates involving second-order information $\nabla^2 f(z)$, in practice they are implemented using double backpropagation. In particular, because those algorithms only require second-order information through Hessian-vector products of the form $\nabla^2 f(z) \nabla f(z)$, these quantities can be efficiently computed by differentiating the function $\nabla f(z)$ in the direction $\nabla f(z)$, see~\citep{Pearlmutter_1994}. To summarize, consensus optimization implements~\eqref{eq:consensus-opt} exactly but for objectives with specific structure ($f$ must be the composition of ``simple'' building blocks whose first/second derivatives are known); whereas the slingshot stepsize schedule approximately implements this update but applies to arbitrary $f$ and only requires black-box access to first-order queries $\nabla f$.

%% file: sections/discussion.tex
\section{Discussion}\label{sec:discussion}

This paper challenges the conventional wisdom that GDA cannot converge in its original form. Our key algorithmic insight for making GDA converge is the use of time-varying, asymmetric, and sometimes negative stepsizes.  This deviates from standard intuition from optimization theory (since negative stepsizes amount to backward movement) as well as monotone operator theory (since asymmetric steps treat the minimization and maximization variables separately). These new conceptual ideas lead to several natural directions for future research, such as the following. 

\paragraph*{Algorithmic building blocks.} The use of time-varying, asymmetric, negative stepsizes not only improves GDA in its original form, but also in some settings can accelerate more sophisticated variants of GDA (with momentum, optimism, extragradients, etc.) for the trivial reason that vanilla GDA is a special case; see~\cref{sec:linear} for concrete examples where this improves prior complexity bounds. Can these new algorithmic ideas be more seamlessly integrated? With these (or other) algorithmic building blocks? The proposed stepsizes converge in a qualitatively different way from standard approaches (see~\cref{fig:diff-behavior}), 
suggesting a possible algorithmic opportunity of combining these different building blocks, e.g., to exploit different types of structure. More broadly, can other algorithms similarly benefit from time-varying parameters, asymmetric updates, and/or negative movement (in a way that is similar to or different from negative momentum~\citep{gidel2019negative})? In practice or in theory? These opportunities were previously overlooked, and we are hopeful that they may become useful building blocks for the design and analysis of min-max optimization algorithms.

\begin{figure}
	\centering
	\includegraphics[width=0.7\linewidth]{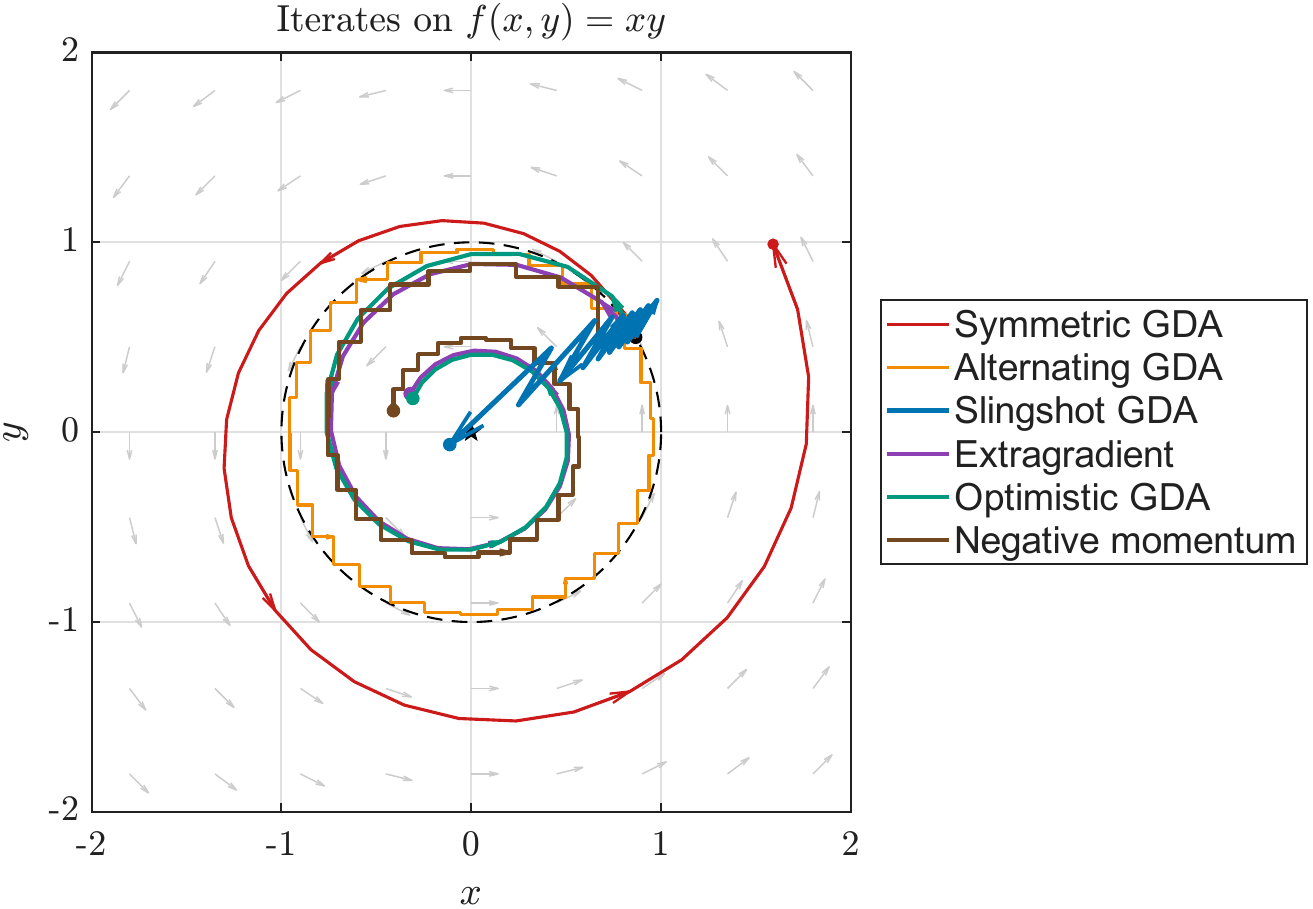}
	\caption{\footnotesize The proposed slingshot schedule converges in a qualitatively different way than existing approaches---suggesting a possible algorithmic opportunity that combines these different building blocks, e.g., to exploit different types of structure.
    Standard approaches (such as extragradients, optimism, and negative momentum) dampen the cyclic behavior of GDA along the grey vector field $(-\nabla_x f, \nabla_y f) = (-y,x)$, leading to trajectories that cycle inwards. In contrast, the slingshot stepsize schedule alternately moves in the orthogonal directions $(\nabla_x f, \nabla_y f) = (y,x)$ and $(-\nabla_x f, -\nabla_y f) = (-y,-x)$. 
    }
	\label{fig:diff-behavior}
\end{figure}

\paragraph*{Convergence in more general settings.} This paper investigates the settings of unconstrained bilinear, quadratic, and (strongly) convex-concave objectives since these are classical counterexamples for GDA and are standard testing grounds for new algorithms. To what extent do the phenomena in this paper extend to more general settings such as non-convex-concave objectives, constrained domains, and/or stochastic gradients? These questions are motivated by applications such as games, control, and deep learning. Can GDA still converge, and if so at what rate? Are the optimal stepsizes qualitatively similar to (any variation of) the slingshot stepsize schedule?

\paragraph*{Faster convergence in restricted settings.} Is faster convergence possible if the objective function is more structured? One can always decompose $f(x,y) = f_{x,y}(x,y) + f_x(x) - f_y(y)$, and in some applications the ``coupled'' term $f_{x,y}$ and/or the ``uncoupled'' terms $f_x,f_y$ have additional structure. One potential direction is if the coupled term $f_{x,y}(x,y) = x^{\top} \bm B y$ is bilinear, which is motivated for example by Fenchel games~\citep{Wang_Abernethy_Levy_2024} and Lagrangian dualization of constrained minimization problems~\citep{Bertsekas_2014}; recent work has designed faster algorithms for example when $\bm B$ is square and full rank \citep{Kovalev_Borodich_2024}. A second potential direction is if the uncoupled terms $f_x$ and $f_y$ are strongly convex and smooth with different parameters; refined accelerated rates are known for other first-order algorithms \citep{Borodich_Kormakov_Kovalev_Beznosikov_Gasnikov_2023, Zhang_Hong_Zhang_2022,Lin_Jin_Jordan_2020_opt, Jin_Sidford_Tian_2022}, and it is natural to ask if these can also be achieved by GDA. A third potential direction is if the Hessians of $f_x$ and $f_y$ have additional spectral structure; such structure has been exploited in (non-min-max) convex optimization~\citep{driscoll1998potential,goujaud2022super,oymak2021provable,kelner2022big}. A fourth potential direction is when $f_x$ and $f_y$ admit computationally tractable proximal operators, which is motivated for example by applications in image reconstruction and signal processing~\citep{Chambolle_Pock_2011,Gabay_Mercier_1976,Glowinski_Marroco_1975}.

\paragraph*{Robustness.} This paper shows that interspersing negative stepsizes---occasionally of very large magnitude---helps GDA converge. We expect that such stepsizes may make the trajectory of GDA more sensitive to rounding errors, incorrect modeling assumptions, and imprecise or stochastic gradient calculations. Indeed, the motivation of the slingshot stepsize schedule is that positive/negative steps of size $h$ cancel to provide a net movement of lower-order size $\Theta(h^2)$ in the direction of convergence, see~\cref{sec:connections}, and this delicate cancellation of higher-order movement could be impacted by modeling error. As a first step towards this, in~\cref{app:stability} we describe a numerically stable fractal-like ordering of the slingshot stepsize schedule for the bilinear setting, by leveraging the connection to convex quadratic optimization in~\cref{obs:paired} and then appealing to known stability ideas in that context. It would be interesting to more broadly understand the tradeoff between convergence rate and stability. As an illustrative example, for unconstrained bilinear problems in the setting of inexact gradients, what is the Pareto-optimal frontier of convergence vs stability---among all GDA stepsize schedules, or more generally among all first-order algorithms? 

\paragraph*{Randomization and derandomization.} In all the different settings, the slingshot stepsize schedule breaks symmetry between the minimization variable $x$ and the maximization variable $y$. Within each pair of iterations, one must choose which variable to first update with a positive step and which to first update with a negative (or zero) step. For min-max problems with linear gradients, this intra-pair order does not affect convergence (assuming exact arithmetic), see~\cref{ssec:linear-alternative}. For problems with nonlinear gradients, we break symmetry by randomizing the order between $x$ and $y$ within every pair of iterations, as this simplifies the design and analysis of the slingshot stepsize schedule, see~\cref{sec:nonlinear}. We conjecture that this result can be de-randomized. Note that this randomization for breaking \emph{intra}-pair symmetry of stepsize signs is distinct from the randomization investigated in~\cref{app:rand} for breaking \emph{inter}-pair symmetry of stepsize magnitudes.

\paragraph*{Implicit bias.} The proof of~\cref{thm:bilinear-ub} not only shows that GDA converges to a saddle point, but moreover isolates which one: it is the closest solution to initialization, namely $x^*$ and $y^*$ are the respective projections of $x_0$ and $y_0$ onto the solution sets $\tilde{x} + \Ker(\bm{B\T})$ and $\tilde{y} + \Ker(\bm{B})$. This can be interpreted as a form of ``implicit bias'' for GDA, and mirrors the well-known fact that GD always converges to the closest solution for convex quadratic minimization. It would be interesting to understand the implicit bias of the slingshot stepsize schedule beyond the bilinear setting.

%% file: sections/app_alternative.tex
\section{Alternative implementations}\label{app:alternative}

Here we illustrate how the slingshot stepsize schedule can be augmented with other ideas from the literature on stepsize schedules for (non-min-max) convex optimization. We consider i.i.d.\ random stepsizes in~\cref{app:rand} and stable fractal orderings in~\cref{app:stability}. For simplicity we focus here on the bilinear setting; it seems plausible that similar ideas may extend to more general settings.

\subsection{Random stepsizes}\label{app:rand}

Here we show how to obtain the same optimal convergence rate in~\cref{thm:bilinear-ub} by using \emph{random} stepsizes. This is inspired by the recent use of i.i.d.\ stepsizes in convex optimization~\citep{pronzato11,pronzato13,kalousek,altschuler2018greed,alt24random}. Like the deterministic stepsizes in~\cref{def:steps-bilinear}, these random stepsizes circumvent the classical GDA counterexamples because they are time-varying, asymmetric, and sometimes negative. One possible benefit of these randomized stepsizes is that they do not require prior knowledge of the number of iterations, and another possible benefit is that the variability in random runs could potentially be exploited for faster parallelized optimization~\citep[\S5]{alt24random}.
\par The proposed random stepsizes are linked in pairs in the same way as the deterministic stepsizes in~\cref{def:steps-bilinear}, i.e., 
\begin{align}\label{eq:opt_schedule-rand}
	\alpha_{2t} = -\alpha_{2t+1} = -\beta_{2t} = \beta_{2t+1} = h_t,\qquad \text{for } t=0,1, \dots \,,
\end{align}
but now $h_t^{-2}$ are not chosen deterministically (as roots of a Chebyshev polynomial), but instead are i.i.d.\,from the Arcsine distribution on the interval $(m,M)$, i.e., the distribution with density 
\begin{align}
	\frac{d\mu}{dz} = \frac{1}{\pi \sqrt{(M-z)(z-m)}} \cdot \bm{1}[z \in (m,M)] \,.
\end{align}

\begin{theorem}[Optimal random stepsizes]\label{thm:rand}
	Consider any dimensions $d_x, d_y$, any initialization $z_0 = (x_0,y_0) \in \R^{d_x} \times \R^{d_y}$, and any bilinear min-max problem~\eqref{eq:bilinear} with spectral bounds $0 < m \leq M < \infty$. Using the aforementioned random stepsizes, GDA converges to a saddle point $z^*$ at rate 
	\begin{equation}
		\|z_{2T} - z^*\| \leq C_T \|z_0 - z^*\|
	\end{equation}
	where $C_T$ satisfies
	\begin{align*}
		\lim_{T \to \infty} (C_T)^{1/T} \overset{\text{a.s.}}{=} \frac{\sqrt{\kappa} - 1}{\sqrt{\kappa + 1}}\,.
	\end{align*}
\end{theorem}
\begin{proof}
	By an identical calculation as in our result for deterministic stepsize schedules (\cref{thm:bilinear-ub}), the positive/negative pairing~\eqref{eq:opt_schedule-rand} of the stepsizes ensures that for every $T$,
	\begin{align*}
		\begin{bmatrix}
			x_{2T} \\
			y_{2T}
		\end{bmatrix}
		=
		\begin{bmatrix}
			p_T(\bm {BB}\T) & \bm 0 \\ \bm 0 & p_T(\bm {B}\T \bm{B} )
		\end{bmatrix}
		\begin{bmatrix}
			x_{0} \\
			y_{0}
		\end{bmatrix}\,,
	\end{align*}
	where $p_T(z) := \prod_{t=0}^{T-1} (1 - h_t^2 z) = \prod_{t=0}^{T-1} (1 - z/r_t)$ and $r_t := h_t^{-2}$ are i.i.d.\ from the Arcsine distribution. By the same diagonalization argument as done there, $\|x_{2T} - x^*\| \leq C_T \, \|x_0 - x^*\|$ and $\|y_{2T} - y^*\| \leq C_T \, \|y_0 - y^*\|$ where $C_T := \sup_{\lambda \in [m,M]} |p_T(\lambda)|$. 
	Now for any $\lambda \in [m,M]$,
	\begin{align*}
		|p_T(\lambda)|^{1/T} 
		=
		\prod_{t=0}^{T-1} |1 - \lambda/r_t |^{1/T}
		=
		\exp\left( \frac{1}{T} \sum_{t=0}^{T-1} \log |1 - \lambda/r_t| \right)
		\overset{\text{a.s.}}{\longrightarrow} \exp\left( \E_{r} \log |1 - \lambda/r| \right)
		= \frac{\sqrt{\kappa} - 1}{\sqrt{\kappa} + 1}\,.
	\end{align*}
	Above the first step is by the definition of $p_T$, the second step is by taking an exponential and logarithm, the third step is by the Law of Large Numbers, and the final step is by the equalizing property of the Arcsine distribution~\citep[Lemma 2.2]{alt24random}.
	The claim follows.
\end{proof}

We conclude with several remarks about these random stepsizes. 1) Optimality of the above asymptotic rate follows immediately from 
taking the limit as $T \to \infty$ of the optimal rate $R_T^{1/T}$ for every finite $T$, see~\cref{lem:extremal-bilinear}. 2) This random stepsize schedule achieves the optimal convergence rate asymptotically, but not for every finite number of iterations $T$. One can prove high probability bounds on this convergence rate for finite $T$ via (non-asymptotic) concentration inequalities rather than using the (asymptotic) Law of Large Numbers, see~\citep[\S5]{alt24random}. 3) The same permutation invariances apply as in the deterministic case, cf.~\cref{ssec:linear-alternative}. 4) These random stepsizes appear to be more sensitive to numerical precision and modeling error than the deterministic stepsize schedule proposed in the following section, and therefore we suggest using the latter in practice.

\subsection{Stable fractal stepsizes}\label{app:stability}

As discussed in~\cref{ssec:linear-alternative}, the convergence rate of GDA for bilinear problems is invariant with respect to several aspects (e.g., permutations) of the slingshot stepsize schedule---assuming exact arithmetic. However, it is a well-known phenomenon that two algorithms which are identical in exact arithmetic can perform quite differently when implemented numerically. In this spirit,~\cref{fig:stability} illustrates how rounding errors (even when using unusually high levels of precision) can make some permutations of the slingshot stepsize schedule extremely unstable. 
\begin{figure}[H]
	\centering
	\includegraphics[width=0.9\linewidth]{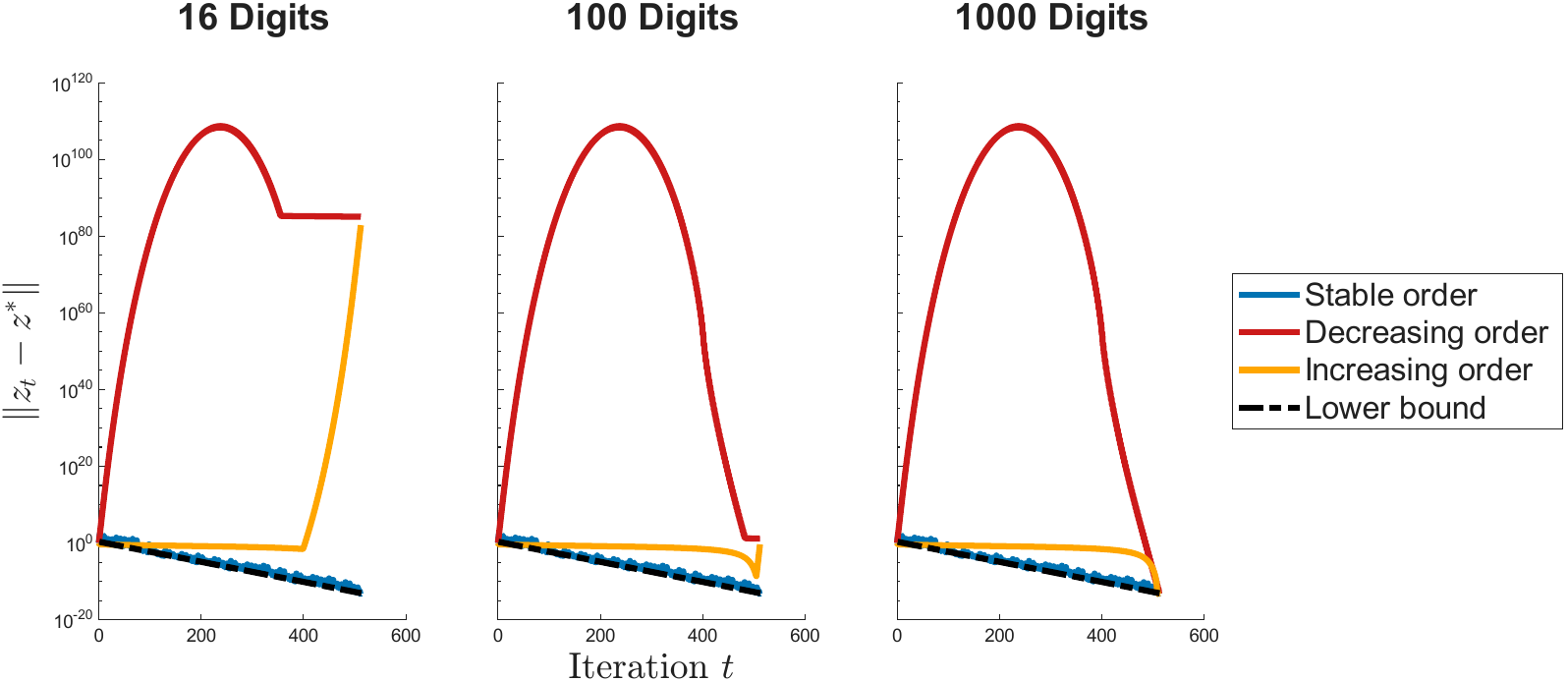}
	\caption{\footnotesize Convergence of different orderings of the proposed steps using different levels of numerical precision (``vpa'' command in MATLAB). We run $512$ iterations on the random bilinear problem described in~\cref{app:experimental-setup}. 
    }
	\label{fig:stability}
\end{figure}

These types of stability issues are well-studied in the literature on GD stepsize schedules for convex quadratic minimization. In particular, fractal permutations of the Chebyshev stepsize schedule are known to ensure better stability in both theory and practice~\citep{lebedev1971ordering, lebedev1973solution, lebedev1976utilization,lebedev2002construction,lebedev2004construction,agarwal2021acceleration}. Because our GDA stepsize schedules emulate these classical GD stepsize schedules (see~\cref{obs:paired} for the precise statement), this immediately suggests the following stable stepsize schedule for GDA.

\begin{defin}[Stable ordering of the slingshot stepsize schedule]
	A more numerically stable version of~\cref{def:steps-bilinear} is 
	\begin{align}
		\alpha_{2t} = -\beta_{2t} = -\alpha_{2t+1} = \beta_{2t+1} = h_{s^T_t}\,, \qquad  t\in\{0,1,\dots, T-1\}\,.
	\end{align}
	Here $\{h_t\}_{t=0}^{T-1}$ are the stepsizes in~\cref{def:steps-bilinear}, but the difference is in the indexing: $s^T$ is the stable ordering proposed in \citep{lebedev1971ordering}, namely initialized with $s^1 = [0]$ and defined recursively as
	\begin{equation}
		s^{2T} := \mathrm{interlace}(s^{T}, 2T - 1 - s^T),
	\end{equation}
	where $\mathrm{interlace}([a_1,...,a_k],[b_1,...,b_k] ) := [a_1,b_1,...,a_k,b_k]$.
\end{defin}

We refer to~\citep{agarwal2021acceleration} for a modern exposition of this stable fractal ordering, as well as visualizations, interpretations, and extensions. In practice, this version of the slingshot stepsize schedule appears to be more stable to rounding errors, see~\cref{fig:stability}. And in theory, due to~\cref{obs:paired}, one can use known stability results for GD in~\citep[\S3.3]{agarwal2021acceleration} for additive error (which can capture rounding error or inexact gradient computation) in order to prove analogous bounds for GDA.

%% file: sections/app_deferred.tex
\section{Deferred details}

\subsection{General second-order approximations}\label{app:second-order}

Here we generalize the informal derivations in~\cref{sec:connections} to arbitrary GDA schedules. 
Let us begin by writing the GDA update more concisely as
\begin{align*}
	z_{2t+1} = z_{2t} + \bm{D_{2t}} \nabla f(z_{2t}) \qquad \text{where} \qquad z_{2t} := \begin{bmatrix}
		x_{2t} \\ y_{2t}
	\end{bmatrix}
	\qquad \text{and} \qquad
	\bm{D_{2t}} := \begin{bmatrix}
		-\alpha_{2t} \bm{I} & \bm{0} \\ \bm{0} & \beta_{2t} \bm{I}
	\end{bmatrix}\,.
\end{align*}
By a similar Taylor expansion as in~\cref{sec:connections}, for sufficiently small stepsizes, 
\begin{align*}
	\nabla f(z_{2t+1}) = \nabla f(z_{2t} + \bm{D_{2t}} \nabla f(z_{2t})) \approx \nabla f(z_{2t}) + \nabla^2 f(z_{2t}) \bm{D_{2t}} \nabla f(z_{2t})\,.
\end{align*}
Hence two steps of GDA is approximately equal to
\begin{align}
	z_{2t+2} &= z_{2t+1} + \bm{D_{2t+1}} \nabla f(z_{2t+1}) \nonumber
	\\ &\approx z_{2t} + \left( \bm{D_{2t}} + \bm{D_{2t+1}} \right) \nabla f(z_{2t}) + \bm{D_{2t+1}} \nabla^2 f(z_{2t}) \bm{D_{2t}} \nabla f(z_{2t})\,. \label{eq:app-general}
\end{align}
This recovers the derivations in~\cref{ssec:connections:linear} and~\cref{ssec:connections:nonlinear} by plugging in the corresponding slingshot stepsize schedules (and taking an expectation for the latter scheme as it is randomized). 
\par To interpret the expansion~\eqref{eq:app-general}, notice that the first movement term $\left( \bm{D_{2t}} + \bm{D_{2t+1}} \right) \nabla f(z_{2t})$ amounts to a single GDA step from $z_{2t}$, with stepsizes $\alpha_{2t} + \alpha_{2t+1}$ and $\beta_{2t} + \beta_{2t+1}$ in the $x$ and $y$ variables, respectively. Recall from~\cref{ssec:intro:slingshot} that property (iii) of slingshot stepsize schedules requires positive net movement $\alpha_{2t} + \alpha_{2t+1}, \beta_{2t} + \beta_{2t+1} \geq 0$ in both variables; thus this GDA movement in~\eqref{eq:app-general} amounts to standard ``forward'' net movement in each variable. The other movement term $\bm{D_{2t+1}} \nabla^2 f(z_{2t}) \bm{D_{2t}} \nabla f(z_{2t})$ is related to Hamiltonian gradient descent; this connection is precise when $\bm{D_{2t+1}}$ and $\bm{D_{2t}}$ are scalings of the identity (in expectation), as is the case for the versions of the slingshot stepsize schedule discussed in~\cref{ssec:connections:nonlinear}.

\subsection{Proof of~\cref{lem:diverge-universal}}\label{app:ub-diverge}

	\paragraph*{(i) Constant stepsizes.}
	In this case, $\GDA$ has linear time-invariant dynamics:
	\begin{equation}
		\begin{bmatrix}x_{t+1}\\ y_{t+1}\end{bmatrix} = \bm  U\begin{bmatrix}x_{t}\\ y_{t}\end{bmatrix} \qquad \text{ where }  \qquad	\bm U = 
		\begin{bmatrix}
			\bm I & -\alpha \bm B \\
			\beta \bm{B}^\top & \bm I
		\end{bmatrix}\,.
		\label{eq:gda-constant}
	\end{equation}
	Let $\bm{B} = \bm{V \Sigma W}^\top$ denote the singular value decomposition of $\bm{B} $, and let $\sigma_1, \dots, \sigma_k$ denote the non-zero singular values. Then, for an appropriate permutation matrix $\bm \Pi$, 
	\begin{align*}
		\bm U = 
		\begin{bmatrix}
			\bm V & \bm 0 \\ \bm 0 & \bm W
		\end{bmatrix}
		\bm \Pi
		\bm D
		\bm \Pi^\top
		\begin{bmatrix}
			\bm V^\top & \bm 0 \\ \bm 0 & \bm W^\top
		\end{bmatrix}
	\end{align*}
	where $\bm D = \mathbb{D}( \bm M_1, \dots, \bm M_k, \bm I)$ is the diagonal concatenation of the $2 \times 2$ matrices
	\begin{align*}
		\bm M_j := \begin{bmatrix}
			1 & -\alpha \sigma_j \\
			\beta \sigma_j & 1 
		\end{bmatrix}
	\end{align*}
	and an identity matrix of the appropriate dimension $d_x + d_y - 2k$. Since the spectral radius is invariant with respect to similarity transforms, and since also the spectral radius of a diagonal concatenation is equal to the largest spectral radius among the constituent diagonal components, it follows that
	\begin{align}
		\rho (\bm U) 
		= \max\{ \rho(\bm M_1), \dots, \rho(\bm M_k), \rho(\bm I) \}
		= \max\{ \abs{1 \plusminus i\sig_1 \sqrt{\alpha \beta}}, \dots, \abs{1 \plusminus i \sig_k\sqrt{\alpha \beta}}, 1\}
		\geq 1\,,
	\end{align}
	where the last inequality is strict if $\alpha, \beta$ are nonzero. Hence $\GDA$ with constant stepsizes does not converge on the bilinear min-max problem~\eqref{eq:bilinear}, and moreover diverges if $\alpha,\beta$ are nonzero.
	
	\paragraph*{(ii) Non-negative stepsizes.} We proceed by a nearly identical diagonalization argument as above, except now the dynamics of $\GDA$ are potentially \emph{time-varying}, with update matrices
	\begin{align*}
		\bm U_t = 
		\begin{bmatrix}
			\bm V & \bm 0 \\ \bm 0 & \bm W
		\end{bmatrix}
		\bm \Pi
		\bm D_t
		\bm{\Pi}^\top
		\begin{bmatrix}
			\bm{V}^\top & \bm 0 \\ \bm 0 & \bm{W}^\top
		\end{bmatrix}
	\end{align*}
	where $\bm D_t = \mathbb{D}( \bm M_{1,t}, \dots, \bm M_{k,t}, \bm I)$ and
	\begin{align*}
		\bm M_{j,t} := \begin{bmatrix}
			1 & -\alpha_t \sigma_j \\
			\beta_t \sigma_j & 1 
		\end{bmatrix}\,.
	\end{align*}
	Since the determinant is invariant with respect to orthogonal transformations and permutations, and since also the determinant of a diagonal concatenation equals the product of the determinants of the constituent diagonal components, it follows that
	\begin{align*}
		\det (\bm U_t)
		=
        \det (\bm {D}_t)
        =
		\prod_{j=1}^k \det ( \bm M_{j,t})
		=
		\prod_{j=1}^k ( 1 + \alpha_t \beta_t \sig_j^2)
		\geq 1\,,
	\end{align*}
	where the last inequality is strict if $\alpha_t,\beta_t$ are strictly positive. Since the spectral radius of a matrix is lower bounded by the geometric mean of the eigenvalues, this implies that $\|\prod_{t=1}^T \bm U_t \| \geq \rho(\prod_{t=1}^T \bm U_t) \geq \prod_{t=1}^T \det( \bm U_t)^{1/(d_x+d_y)} \geq 1$. Thus $\GDA$ with non-negative stepsizes does not converge on the bilinear min-max problem~\eqref{eq:bilinear}.
	
	\paragraph*{(iii) Symmetric stepsizes.} The argument is identical to case (ii), with one modification: the reason that $\det(\bm{M_{j,t}}) \geq 1$ is now due to $\alpha_t = \beta_t \in \R$ rather than $\alpha_t,\beta_t \geq 0$.

\subsection{Proof of~\cref{thm:bilinear-lb}}\label{app:bilinear:lb}

This lower bound for \emph{asymmetric} Krylov-subspace algorithms follows by a straightforward adaptation of the known lower bounds for \emph{symmetric} Krylov-subspace algorithms~\citep{azizian2020accelerating, Ibrahim_Azizian_Gidel_Mitliagkas_2020}. We begin by recalling the definition of these classes of algorithms. Symmetric Krylov-subspace algorithms~\citep{azizian2020accelerating, Ibrahim_Azizian_Gidel_Mitliagkas_2020} are iterative algorithms that satisfy the ``linear span assumption'' 
\begin{align*}
	\begin{bmatrix}
		x_t \\ y_t
	\end{bmatrix}
	\in \begin{bmatrix}
		x_0 \\ y_0
	\end{bmatrix}
	+ \spann \left\{
	\begin{bmatrix}
		- \nabla_x f(x_s,y_s) \\ \nabla_y f(x_s,y_s)
	\end{bmatrix}
	\right\}_{s \in \{0, \dots, t-1\}}
	\,.
\end{align*}
This captures many common algorithms, but precludes asymmetric updates to the minimization variable $x$ and the maximization variable $y$, and thus does not capture the slingshot stepsize schedule (among other algorithms). Such updates are captured by the following richer class of algorithms.

\begin{defin}[Asymmetric Krylov-subspace algorithms]\label{def:krylov}
		An iterative algorithm for a min-max optimization problem is an asymmetric Krylov-subspace algorithm if for all $t$, its $t$-th iterate satisfies
		\begin{align*}
			&x_t \in x_0 + \spann\{\nabla_x f(x_s,y_s)\}_{s \in \{0, \dots, t-1\}} \,,
			\\
			&y_t \in y_0 + \spann\{\nabla_y f(x_s,y_s)\}_{s \in \{0, \dots, t-1\}}\,.
		\end{align*}
	\end{defin}

	We now turn to proving~\cref{thm:bilinear-lb}. Briefly, the argument reduces the question of lower bounding the convergence rate of symmetric Krylov-subspace algorithms for bilinear min-max optimization problems, to lower bounding the maximal value of a low-degree polynomial over an interval---a standard result for lower bounds in quadratic optimization (see, e.g.,~\citep{nesterov-survey}). In the asymmetric case, this reduction is made possible by the following lemma, which can be shown by induction; see~\citep[Appendix C.1]{yoon2021accelerated}. Note that this is slightly more involved than the symmetric case, because in that case the $2 \times 2$ block matrix of polynomials in this lemma can be simplified to a polynomial of the $2 \times 2$ block matrix $\begin{bmatrix} \bm{I} & \bm{B} \\ -\bm{B}^T & \bm{I} \end{bmatrix}$.

	\begin{lemma}\label{lem:block-poly}
		For every iteration $t$, the $t$-th iterate of an asymmetric Krylov-subspace algorithm on the bilinear problem~\eqref{eq:bilinear} is of the form
		\begin{equation} \label{eq:poly_cond}
			\begin{bmatrix} x_t - x^*\\ y_t - y^* \end{bmatrix}=
			\begin{bmatrix}
				p^{(11)}(\bm{B}\bm{B}\T) & p^{(12)}(\bm{B}\bm{B}\T) \bm B \\
				p^{(21)}(\bm{B}\T\bm{B}) \bm B\T  & p^{(22)}(\bm{B}\T\bm{B})
			\end{bmatrix} \begin{bmatrix} x_0 - x^*\\ y_0 - y^* \end{bmatrix}\,,
		\end{equation}
		where $p^{(11)}$ and $p^{(22)}$ are polynomials of degree at most $\lfloor t/2\rfloor$,  $p^{(12)}$ and $p^{(21)}$ are polynomials of degree at most $\lceil t/2 \rceil - 1$, and $p^{(11)}(0)=p^{(22)}(0)=1$.
	\end{lemma}

    For simplicity, we prove~\cref{thm:bilinear-lb} below for non-adaptive Krylov-subspace algorithms, i.e., algorithms for which the coefficients of the linear combination in~\cref{def:krylov} do not depend on the instance. This can be upgraded to adaptive algorithms by appealing to a stronger lower bound for convex quadratic optimization, namely rather than appealing to~\cref{lem:extremal-bilinear} (which can only be used to show a lower bound for non-adaptive algorithms for convex quadratic optimization), instead appeal to~\citep[Theorem 5.1]{altschuler2018greed} (which shows that the convergence rate $R_T$ is also optimal for adaptive algorithms).\footnote{The classical lower bound of~\citep{nemirovskij1983problem} could also be used in lieu of~\citep[Theorem 5.1]{altschuler2018greed}, but that only matches $R_T$ asymptotically and is weaker for every finite $T$.} The proof is otherwise the same, as it relies on the same reduction. 

    \begin{proof}[Proof of~\cref{thm:bilinear-lb}]
       Consider a $1$-dimensional bilinear min-max problem $\min_{x \in \R} \max_{y \in \R} \, b(x-x^*)(y-y^*)$, where $y^* = y_0$, $x^* \neq x_0$, and $b^2 \in [m,M]$ is a parameter to be chosen. By~\cref{lem:block-poly}, $x_{2T} - x^* = p(b^2) \cdot (x_0 - x^*)$ for some polynomial $p$ of degree at most $T$ satisfying $p(0) = 1$. By~\cref{lem:extremal-bilinear}, there exists $\lambda \in [m,M]$ satisfying $|p(\lambda)| \geq R_T$. Thus, by choosing $b := \sqrt{\lambda}$, we have $|x_{2T} - x^*| \geq R_T |x_0 - x^*|$. The result now follows since $|y_{2T} - y^*| \geq 0 = |y_0 - y^*|$ holds trivially.
    \end{proof}

\subsection{Proof of~\cref{lem:2step}}\label{app:2-step}

Here we give an alternative, sharper analysis of the two-step progress guarantee in~\cref{lem:2step}. The purpose of including this is twofold. First, it establishes the constants stated in~\cref{lem:2step}, rather than the weakened constants proved in~\cref{ssec:2step}. Second, and more conceptually, this proof is guided by the performance estimation problem (PEP) framework, which systematically combines a given set of valid inequalities in order to obtain the tightest possible analysis of a prescribed type. 
In the present setting of~\cref{lem:2step}, the PEP analysis certifies the desired $2$-step progress guarantee by identifying and combining an appropriate set of inequalities---based on smoothness, (strong) convexity-concavity, and sum-of-squares terms, some evaluated at auxiliary non-iterate points.
This ability of PEP to systematically search over convergence proofs was helpful for our design of the slingshot stepsize schedule and directly motivated the simpler, more interpretable proof in~\cref{ssec:2step}. Indeed, the proof in~\cref{ssec:2step} can be viewed as the same underlying argument with this detailed PEP/sum-of-squares certificate replaced by slightly weaker Lipschitz estimates. Notice also that both proofs crucially rely on the same non-iterate mixed points identified by PEP, such as $(x_1^{(-)},y_1^{(+)})$; the use of these auxiliary points is essential for any such analysis, as discussed below.

Recall the notation in~\cref{ssec:2step}. Here we prove the stronger inequality
\begin{equation}\label{eq:2step-pf-tight}
\left(1- h \mu\right) \|z_0 - z^*\|^2 - c_{h,\mu} \left\|\nabla f(z_0)\right\|^2 
- \frac{1}{2} \left( \left\|z_{2}^{(-)}-z^*\right\|^2 + \left\|z_{2}^{(+)}-z^*\right\|^2 \right) \geq 0\,,
\end{equation}
where $c_{h,\mu} :=\tfrac{h^2\left(1 +(4\mu-5)h + (6-9\mu+2 \mu^2) h^2\right)}{2(1+(\mu-2)h)}$.
Note that~\eqref{eq:2step-pf-tight} implies~\cref{lem:2step} because for $h < 1/3$, then $c_{h,\mu}$ is increasing in $\mu \in [0,1]$, and takes the desired minimal value $c_{h,0} = h^2(1-3h)/2$ at $\mu = 0$.

\paragraph*{Non-negative building blocks.} We prove~\eqref{eq:2step-pf-tight} by re-writing the left hand side as the sum of three types of non-negative quantities. The simplest such non-negative quantity is based on the assumption that $f$ is $1$-smooth, i.e., $\nabla f$ is $1$-Lipschitz:
\begin{equation}
P(z,z') := \|z - z'\|^2 - \|\nabla f(z) - \nabla f(z')\|^2
\end{equation}
is non-negative for any two points $z = (x,y)$ and $z' = (x',y')$. A second type of non-negative quantity that we use is based on the co-coercivity inequality for a convex function $g$ (which will later be taken to be $f(\cdot,y)$ for a fixed $y$, or $-f(x,\cdot)$ for a fixed $x$):
\begin{align*}
    Q_g(v,w) := \;g(v) - g(w)  + \frac{1}{2(1 - \mu)} \Big[2  \langle \nabla g(w) - \mu \nabla g(v), w - v \rangle - \|\nabla g(v) - \nabla g(w)\|^2 - \mu \|v - w\|^2  \Big]
\end{align*}
is non-negative for every pair of points $v,w$~\citep{taylor2017smooth}. The final type of non-negative quantity that our proof uses is a sum-of-squares term, with constituent squares
\begin{align*}
S_1 &= \frac{1}{2(1-\mu)}\left\|h\mu \nabla_x f(x_0,y_0) - \nabla_x f\left(x^{(+)}_1, y^{(+)}_1 \right) + \nabla_x f\left(x^*, y^{(+)}_1\right) + \mu x_0 - \mu x^* \right\|^2\,,\\
 S_2 &= \frac{1}{2(1-\mu)}\left\|h\mu \nabla_y f(x_0,y_0) - \nabla_y f\left(x^{(-)}_1, y^{(-)}_1 \right) + \nabla_y f\left( x^{(-)}_1,y^*\right) -  \mu y_0 + \mu y^* \right\|^2\,, \\
S_3 &= \frac{1}{2(1-\mu)}\left\| h \mu \nabla_x f(x_0,y_0) + \nabla_xf \left(x_1^{(-)},y^* \right) -\mu x_0 +\mu x^*  \right\|^2\,, \\
S_4&= \frac{1}{2(1-\mu)} \left\| h \mu \nabla_y f(x_0,y_0) + \nabla_yf \left(x^*,y_1^{(+)} \right) + \mu y_0  -\mu y^* \right\|^2\,, \\
S_5 &= \frac{1}{2\rho(1-\mu)} \left\|\rho \nabla_x f\left(x^{(+)}_1, y^{(+)}_1\right) - \nabla_x f\left(x^{(-)}_1, y^{(+)}_1\right) - 2h\mu \nabla_x f(x_0, y_0)  \right\|^2\,,\\
S_6 &= \frac{1}{2\rho(1-\mu)} \left\|\rho \nabla_y f\left(x^{(-)}_1, y^{(-)}_1\right) - \nabla_y f\left(x^{(-)}_1, y^{(+)}_1\right) - 2h \mu \nabla_y f(x_0, y_0)  \right\|^2\,, \\ 
S_7 &= \frac{1}{2\rho(\rho-h)} \left\| (\rho-h) \nabla f(x^{(-)}_1, y^{(+)}_1)  - (\rho + 2h(h-1)) \nabla f(x_0,y_0) \right\|^2\,,
\end{align*}
where we denote $\rho := 1+h\mu-h$. 

\paragraph*{Certificate of non-negativity.} Specifically, the left hand side of~\eqref{eq:2step-pf-tight} is equal to $h$ times
\begin{align}
Q_{\phi}\left(x_{1}^{(+)}, x_{1}^{(-)}\right) + &Q_{\phi}\left(x^*, x_{1}^{(+)}\right) + Q_{\phi^*}\left(x_{1}^{(-)}, x^* \right)  + Q_{\psi}\left(y_{1}^{(-)}, y_{1}^{(+)}\right) + Q_{\psi}\left(y^*, y_{1}^{(-)}\right) +   Q_{\psi^*}\left(y_{1}^{(+)}, y^*\right) \nonumber \\
+ &\frac{1}{2}
P\left(z_0, z'\right)
+ \sum_{i=1}^7 S_i\,, \label{eq:2step-identity}
\end{align}
where we denote $z' := (x_{1}^{(-)},y_{1}^{(+)})$, $\phi := f(\cdot,y_1^{(+)})$, $\phi^* = f(\cdot,y^*)$, $\psi := -f(x_1^{(-)},\cdot)$, and $\psi^* := -f(x^*,\cdot)$.

The fact that the left hand side of~\eqref{eq:2step-pf-tight} is equal to $h$ times~\eqref{eq:2step-identity} can be checked in a conceptually simple (albeit tedious) manner by expanding out the aforementioned definitions for the three types of terms and plugging in the definitions~\eqref{eq:2step-traj1} and~\eqref{eq:2step-traj2} of the algorithm's iterates. For brevity and convenience, we provide a short Mathematica script that rigorously verifies this identity at the URL given in the references~\citep{MathematicaURL}. 

\par It follows that~\eqref{eq:2step-pf-tight} is non-negative, as desired, since~\eqref{eq:2step-identity} is the sum of non-negative terms. Note that here we are using the assumption that $f$ is convex-concave because that implies $\phi,\psi,\phi^*,\psi^*$ are convex functions, hence the co-coercivities in~\eqref{eq:2step-identity} are non-negative.

\paragraph*{Remarks.} We make two remarks about this proof. First, the expansion~\eqref{eq:2step-identity} enjoys a certain duality: appearances of $x_{t}^{(-)}$ and $x_t^{(+)}$ are respectively complemented by appearances of $y_{t}^{(+)}$ and $y_t^{(-)}$. This occurs in the co-coercivities (note the symmetry between the first three terms and the second three terms), the smoothness term (note the duality in $z' = (x_1^{(-)}, y_1^{(+)})$), and in the sum-of-squares terms (note that $S_1,S_2$ and $S_3,S_4$ and $S_5,S_6$ are paired in this way; $S_7$ is the combination of two such terms as it involves the full-vector gradient).

Second, this analysis approach---expressing the desired quantity in terms of non-negative quantities involving co-coercivities, smoothness inequalities, and sum-of-squares terms---is motivated by the Performance Estimation Problem (PEP) framework pioneered by~\citep{drori2014performance}. PEP has been recently used in many min-max settings~\citep{Ryu_Taylor_Bergeling_Giselsson_2020, Rubbens_Bousselmi_Colla_Hendrickx_2023}. However, typical PEP approaches rely on (variations of) the elegant fact from (non-min-max) convex optimization that in order to prove the tightest possible convergence rates for a first-order algorithm, it is sufficient to only use function and gradient evaluations at the algorithm's iterates~\citep{taylor2017smooth}. In contrast, our analysis here---and any alternative analysis such as that in~\cref{ssec:2step}---must also incorporate function and gradient evaluations that are \emph{not} at the algorithm's iterates. For example, function and gradient evaluations at $(x_1^{(-)}, y_1^{(+)})$ appear in the sum-of-squares terms, the co-coercivity terms, and the smoothness term.

\subsection{Experimental setup}\label{app:experimental-setup}

\cref{fig:conv,fig:stability} are preliminary numerical simulations that illustrate the convergence and stability properties, respectively, of the slingshot stepsize schedule on bilinear problems. The problem instances $\min_x \max_y x^T \bm{B} y$ are generated by drawing $\tilde{\bm B} \in \R^{30 \times 30}$ with i.i.d.\ uniform entries in $[0,1]$, and setting $\bm B = \bm U(\bm\Sigma+\bm I)\bm V^\top$ where $\tilde{\bm B} = \bm U \bm \Sigma \bm V^\top $ denotes the SVD decomposition of $\tilde{\bm B}$. This ensures that the singular values of $\bm B$ are lower bounded by $m = 1$, and in the problem instance shown we use an upper bound $M = 300$ (i.e., $L := \sqrt{M} = \sqrt{300}$). Qualitatively similar results are observed in other problems. 

\par The stepsizes for extragradient, optimistic GDA, and extra anchored gradient are respectively set to the standard prescriptions of $1/(\sqrt{2}L)$~\citep{gorbunov2022extragradient}, $1/(3L)$~\citep{Gorbunov_Taylor_Gidel},  and $1/(8L)$~\citep{yoon2021accelerated}. For a fair comparison, every full step of the extragradient and extra anchored gradient algorithms is counted as two iterations in~\cref{fig:conv,fig:cc-convergence} because it uses twice as many gradient queries.

\par In these experiments, the practical performance of the compared algorithms matches their theoretical bounds. For example, the slow convergence of extragradient and optimistic GDA in~\cref{fig:conv} illustrates the known rate $\mathcal{O}(\kappa \log 1/\eps)$, and the slingshot stepsize schedule converges at the optimally accelerated rate $\mathcal{O}(\sqrt{\kappa} \log 1/\eps)$ shown in~\cref{thm:bilinear-ub}. As another example, extragradient, optimistic GDA, and the (non-tailored version of) slingshot stepsizes for convex-concave problems match their theoretical rates of $\mathcal{O}(1/\eps)$ in~\cref{fig:cc-convergence}, and the convergence of extra anchored gradient matches the accelerated rate of $\mathcal{O}(1/\sqrt{\eps})$ for smooth convex-concave problems.